\documentclass{article}
\usepackage[a4paper, total={6.5in, 9.5in}]{geometry}
\let\originalleft\left
\let\originalright\right
\renewcommand{\left}{\mathopen{}\mathclose\bgroup\originalleft}
\renewcommand{\right}{\aftergroup\egroup\originalright}
\usepackage{amsfonts}
\usepackage{amsmath}
\DeclareMathOperator{\Tr}{Tr}
\usepackage{amsthm}
\usepackage{caption}
\usepackage{mathtools}
\usepackage{subcaption}
\newtheorem{theorem}{Theorem}[section]
\newtheorem{remark}[theorem]{Remark}
\newcommand{\footremember}[2]{\footnote{#2}\newcounter{#1}\setcounter{#1}{\value{footnote}}}
\newcommand{\footrecall}[1]{\footnotemark[\value{#1}]}

\begin{document}
\title{A Stochastic Model of Intracellular Calcium Concentrations}
\author{Tom{\'a}s Caraballo\footremember{1}{Department of Mathematics, Wenzhou University, Wenzhou, Zhejiang Province, Wenzhou, 325035, P.R.China.} \footremember{2}{Departamento de Ecuaciones Diferenciales y An{\'a}lisis Num{\'e}rico, Universidad de Sevilla, c/ Tarfia s/n, Seville, 41012, Spain.} \and Macarena G{\'o}mez-M{\'a}rmol\footrecall{2} \and Ignacio Rold{\'a}n\footrecall{2} \footremember{3}{Corresponding author (iroldan@us.es).}}
\date{\today}
\maketitle
\begin{abstract}
In this paper, we study a slow-fast model of Intracellular Calcium Concentrations (ICC) within a stochastic framework, where stochastic perturbations are introduced to capture the intrinsic randomness of such biological systems. We first analyse the three-dimensional single-cell model, based on the FitzHugh-Nagumo structure, and subsequently extend the analysis to a coupled two-cell six-dimensional setting. We prove existence and uniqueness of solutions and establish positivity of the calcium variable. We obtain explicit mean-square stability estimates around deterministic equilibria and quantify how they depend on the noise intensity. Numerical simulations for both the single-cell and coupled systems illustrate relevant qualitative behaviours and reveal additional dynamical features arising in the stochastic framework.
\end{abstract}
\paragraph{Subjclass:} 60H10, 65C30, 92C20.

\section*{Introduction}
At present, the scientific community is increasingly concerned with modelling realistic phenomena through differential equations. Many real-world processes, such as population dynamics, chemical reactions, and neuronal activity, can be described within this framework. In particular, in neuroscience, variability is an intrinsic feature of the dynamics and cannot be fully understood within purely deterministic models; see \cite{Faisal2010, Laing2010}. This naturally motivates the incorporation of stochastic effects into the modelling process.

A classical way to introduce randomness into deterministic systems is through the addition of white noise, leading to stochastic differential equations (SDEs). Since the foundational works of It{\^o} \cite{Ito1951_1, Ito1951_2}, SDEs have become a standard tool for modelling systems subject to random perturbations; see for example \cite{Hairer2010, Lord2014}. Existence and uniqueness results for SDEs can be found in classical references such as \cite{Karatzas1991, ksendal2003}. In contrast to ordinary differential equations, solutions of SDEs are stochastic processes, and their analysis requires probabilistic tools and suitable notions of stochastic convergence.

Other types of perturbations of excitable slow-fast systems may also be considered in order to obtain more realistic descriptions. For instance, the introduction of time delay or random perturbations within a deterministic framework has been investigated in \cite{Rold2026}. In that setting, randomness is incorporated through random differential equations, whereas in the present work we adopt a stochastic formulation based on stochastic differential equations, which leads to substantially different analytical and numerical challenges.

From the numerical point of view, the approximation of SDEs differs substantially from the deterministic case due to stochastic integrals and stochastic differential terms. Classical methods include the Euler-Maruyama and Milstein schemes \cite{Cambanis1996, Maruyama1955}, as well as higher-order integrators based on the It{\^o}--Taylor expansion \cite{Ewald2005}, see also \cite{Milstein2004, Ruemelin1982}. These methods provide satisfactory results whenever the underlying problems are not stiff. However, for stiff problems, as in the deterministic setting, higher-order methods with good stability properties become necessary. Constructing strong order two methods requires additional analysis, since stochastic orders are often reduced compared to their deterministic counterparts \cite{Kloeden1992, Mao2007, Talay1990}.

Slow-fast systems form a particular class of stiff problems; see, for instance, \cite{Hairer2010}. In the stochastic setting, time-scale separation combined with noise may modify stability properties and lead to qualitatively different dynamical behaviour; see, for instance, \cite{Kuehn2015}. It is therefore necessary to analyse the behaviour of their stochastic solutions, distinguishing between cases with stable and unstable equilibria; see \cite{Berglund2003, Berglund2006, Berglund2010, Berglund2012_1, Galtier2012}. Even when the deterministic system possesses a globally asymptotically stable equilibrium, stochastic perturbations may lead to solutions fluctuating around that equilibrium in an average sense; see representative asymptotic results in \cite{Berrhazi2018, Caraballo2018}.

A well-known deterministic model for excitable systems is the FitzHugh-Nagumo model, a simplified version of the Hodgkin-Huxley system; see for example \cite{Guckenheimer2009, Hastings1975, Roc2000}. It describes the behaviour of excitable cells such as neurons. A partial differential equation version is studied in \cite{Carter2015}. Stochastic perturbations of FitzHugh-Nagumo-type systems have been widely investigated, including models with additive noise \cite{Bonaccorsi2008, Yamakou2019}, moment analysis \cite{Tuckwell1998}, stochastic partial differential equation formulations \cite{Eichinger2022, Tuckwell2008}, and models with different noise sources acting on distinct variables \cite{Berglund2012_2}.

Based on the FitzHugh-Nagumo structure, a three-dimensional slow-fast model for Intracellular Calcium Concentrations (ICC) in single neurons was introduced in \cite{Krupa2013}. This model extends the classical excitable dynamics by incorporating an additional variable describing calcium regulation while preserving the slow-fast structure. In \cite{Bandera2022, Bandera2026, Fern2020} can be found deterministic extensions and numerical studies, including coupled-cell variants.

The aim of this paper is to study stochastic versions of the ICC model obtained by adding white-noise perturbations that preserve its slow-fast structure. We first analyse the three-dimensional single-cell model and then extend the analysis to a coupled two-cell six-dimensional system. We establish existence and uniqueness of solutions, analyse qualitative properties of the stochastic system, including positivity of the calcium variable and its behaviour near deterministic equilibria, and investigate numerically the impact of stochastic forcing on the dynamics.

To deal with stiffness and multiscale effects, we employ a stochastic version of the TR-BDF2 method, whose deterministic counterpart was originally introduced and analysed in \cite{Bank1985, Hairer2010, Hosea1996}. The stochastic extension developed in \cite{Caraballo2026} attains strong order two while preserving the favourable stability features of the deterministic method. The numerical implementation follows standard MATLAB-based approaches for stiff problems \cite{Alzubaidi2010, Laing2010}.

The remainder of the paper is organised as follows. Section~\ref{sec2} presents the deterministic ICC model and recalls its main properties. Section~\ref{sec3} introduces the stochastic formulation and establishes the main theoretical results concerning well-posedness, positivity, and asymptotic behaviour of the stochastic solutions. Section~\ref{sec4} is devoted to numerical simulations illustrating the theoretical findings and highlighting qualitative behaviours induced by stochastic perturbations. Finally, Section~\ref{sec5} extends the analysis to the coupled two-cell system and illustrates the new dynamical features induced by stochastic perturbations.

\section{Problem Formulation}\label{sec2}
In this section we recall the single-cell deterministic ICC model (three-dimensional), which will serve as the reference configuration for the stochastic analysis carried out later. This single-cell model will also provide the building block for the coupled two-cell (six-dimensional) extension introduced in a subsequent section. We also introduce the notation used throughout the paper.

The ICC model is a slow-fast system based on the well-known FitzHugh-Nagumo model (FHN). Following \cite{Krupa2013}, the model reads:
\begin{equation}
\left\{\begin{array}{l}
\displaystyle \dot{x} = \tau\left(b_{0}y + b_{1}x - x^{3} - \phi_{f}\left(z\right)\right),\\
\displaystyle \dot{y} = \tau\varepsilon k\left(a_{0}x + a_{1}y + a_{2}\right),\\
\displaystyle \dot{z} = \tau\varepsilon\left(\phi_{r}\left(x\right) - \frac{z - z_{b}}{\tau_{z}}\right),\\
\displaystyle \left(x\left(0\right), y\left(0\right), z\left(0\right)\right)^{\top} \text{ is given},
\end{array}\right.
\label{ICC}
\end{equation}
\noindent where
\begin{equation*}
\begin{array}{ll}
\displaystyle \phi_{f}\left(z\right) = \frac{\mu z}{z + z_{0}}, & \displaystyle \phi_{r}\left(x\right) = \frac{\lambda}{1 + \exp\left(- \rho\left(x - x_{on}\right)\right)}.
\end{array}
\end{equation*}

Here, the overdot stands for the derivative with respect to $t$, and $T > 0$ is the final time, i.e., $t \in \left[0, T\right]$. Parameter $\tau > 0$ is introduced in \cite{Krupa2013} and it rescales the time interval. Parameter $\varepsilon$ is taken as $0 < \varepsilon \ll 1$ and therefore it separates the timescales between the fast variable $x$ and the slow ones $y$ and $z$. Variable $x$ represents the cell electrical activity, variable $y$ represents recovery and variable $z$, the intracellular calcium level, acts as a feedback onto the dynamics of $x$ through the Hill function $\phi_{f}\left(z\right)$, which is bounded by $\mu > 0$. Parameters $z_{0}, k, \lambda, \rho, x_{on}$ and $\tau_{z}$ are assumed to be strictly positive, with $k$ being a parameter that may introduce heterogeneity in multi-cell settings. In the single-cell model considered here, we set $k = 1$. The dynamics of variable $z$ is mostly given by $x$ through the sigmoidal function $\phi_{r}\left(x\right)$, which takes values between 0 and $\lambda$. When $\phi_{r}\left(x\right)$ is inactive (close to 0), $z$ decreases to a quasi steady state close to the baseline of the intracellular calcium level, represented by $z_{b} > 0$. This occurs with an exponential decay rate, determined by the ratio $\frac{\tau\varepsilon}{\tau_{z}}$.

On the other hand, parameters $b_{1} > 0$ and $b_{0} < 0$ determine the cubic structure of the $x$-nullcline, while $a_{0} > 0$ (typically 1), $- 1 \ll a_{1} < 0$ and $a_{2} > 0$ define the $y$-nullcline. This cubic function admits a local minimum at $x = - x_{f} < 0$ and a local maximum at $x = x_{f} > 0$. We take $b_{1} = 4$ and $b_{0} = - 1$, as in \cite{Bandera2022, Krupa2013}.

Note that we are interested in the selection of parameters that yields a unique equilibrium $\left(x^{\ast}, y^{\ast}, z^{\ast}\right)$ of the ICC model \eqref{ICC}. From now on we will suppose that we are in this framework. Therefore, we keep the values used in \cite{Bandera2022}:
\begin{equation*}
\begin{array}{lllllll}
\displaystyle \tau = 37, & \displaystyle b_{0} = - 1, & \displaystyle b_{1} = 4, & \displaystyle z_{0} = 5, & \displaystyle \varepsilon = 0.06, & \displaystyle k = 1, & \displaystyle a_{0} = 1,\\
\displaystyle a_{1} = - 0.1, & \displaystyle a_{2} = 0.8, & \displaystyle \lambda = 1.75, & \displaystyle \rho = 4.5, & \displaystyle x_{on} = - 0.45, & \displaystyle z_{b} = 1, & \displaystyle \tau_{z} = 2,
\end{array}
\end{equation*}
\noindent and take $\mu \in \left\{2.21, 2.4, 2.5\right\}$ and the final time $T = 70$, as in \cite{Fern2020}.

Then, to describe the behavior of the solution, we define the $x$-nullcline
\begin{equation*}
S \equiv \left\{- b_{0}y = b_{1}x - x^{3} - \phi_{f}\left(z\right)\right\},
\end{equation*}
\noindent which has two fold curves
\begin{equation*}
\Gamma^{-} \equiv \left\{x = - x_{f}, - b_{0}y = - b_{1}x_{f} + x_{f}^{3} - \phi_{f}\left(z\right)\right\},
\end{equation*}
\begin{equation*}
\Gamma^{+} \equiv \left\{x = x_{f}, - b_{0}y = b_{1}x_{f} - x_{f}^{3} - \phi_{f}\left(z\right)\right\}.
\end{equation*}

Also, we shall denote $S_{l}$ as the sheet of $S$ that lies to the left of $\Gamma^{-}$, $S_{r}$ as the sheet of $S$ that lies to the right of $\Gamma^{+}$ and $S_{m}$ as the sheet of $S$ that lies in the middle of them, i.e.,
\begin{equation*}
S_{l} \equiv \left\{x < - x_{f}, - b_{0}y = b_{1}x - x^{3} - \phi_{f}\left(z\right)\right\},
\end{equation*}
\begin{equation*}
S_{r} \equiv \left\{x > x_{f}, - b_{0}y = b_{1}x - x^{3} - \phi_{f}\left(z\right)\right\},
\end{equation*}
\begin{equation*}
S_{m} \equiv \left\{x \in \left(- x_{f}, x_{f}\right), - b_{0}y = b_{1}x - x^{3} - \phi_{f}\left(z\right)\right\}.
\end{equation*}

Depending on the position of the unique equilibrium point of \eqref{ICC}, three main behaviors of the solution can be identified, depending on $\mu$ (see \cite{Krupa2013}):

\begin{enumerate}
\item The equilibrium point lies on $S_{l}$ or $S_{r}$: then, the equilibrium is stable and, after the solution reaches the sheet where the equilibrium point is, it remains permanently in the steady state, in the vicinity of the equilibrium. This case corresponds to $\mu > 2.45$ (Figure~\ref{Fig1a}, with $\mu = 2.5$).

\item The equilibrium point lies on $S_{m}$ and far from the fold curves $\Gamma^{-}$ and $\Gamma^{+}$: then, the equilibrium point is unstable and there exists an attractive relaxation limit cycle. This case corresponds to $\mu < 2.26$ (Figure~\ref{Fig1b}, with $\mu = 2.21$).

\item The equilibrium point lies on $S_{m}$ but close to a fold curve: then, the equilibrium point is unstable but the system exhibits Mixed-Mode Oscillations (MMOs). This case corresponds to $\mu \in \left[2.26, 2.45\right]$ (Figure~\ref{Fig1c}, with $\mu = 2.4$).
\end{enumerate}

\begin{figure}[ht]
\centering
\begin{subfigure}[ht]{0.45\textwidth}
\centering
\includegraphics[scale=0.30]{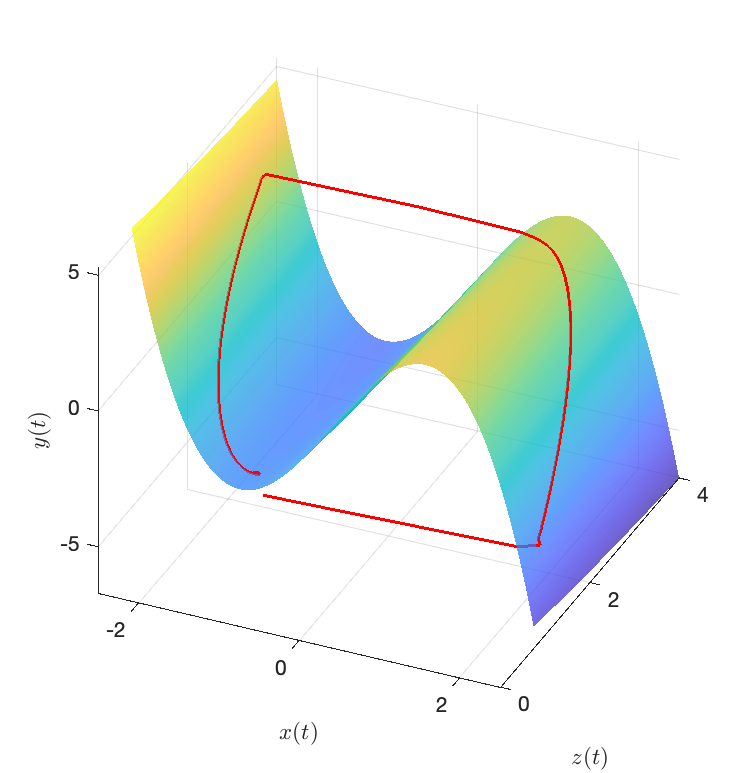}
\caption{Stable equilibrium, attractive case.}\label{Fig1a}
\end{subfigure}
\hfill
\begin{subfigure}[ht]{0.45\textwidth}
\centering
\includegraphics[scale=0.30]{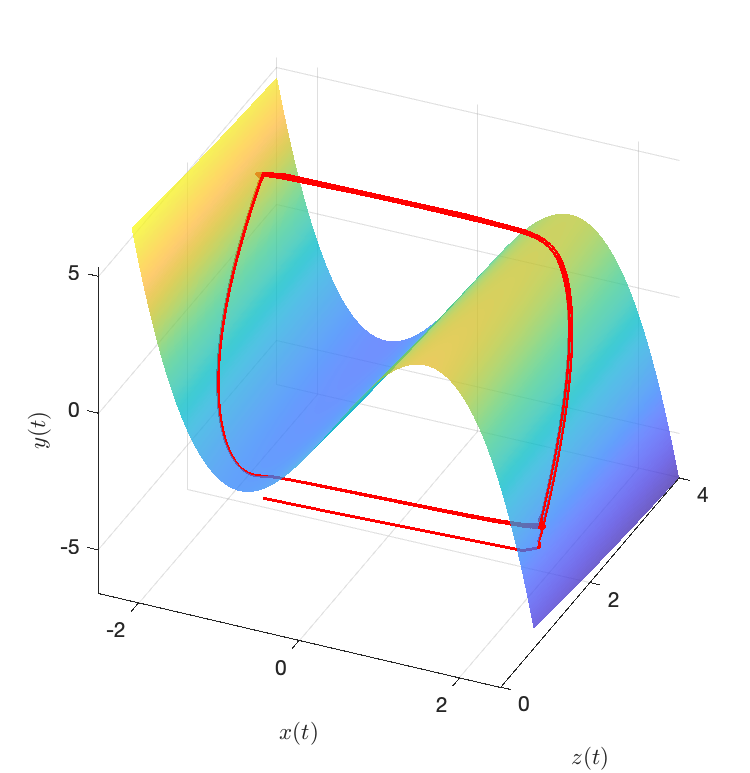}
\caption{Unstable point, relaxation limit cycle.}\label{Fig1b}
\end{subfigure}
\end{figure}
\begin{figure}\ContinuedFloat
\begin{subfigure}[ht]{0.95\textwidth}
\centering
\begin{subfigure}[ht]{0.45\textwidth}
\centering
\includegraphics[scale=0.30]{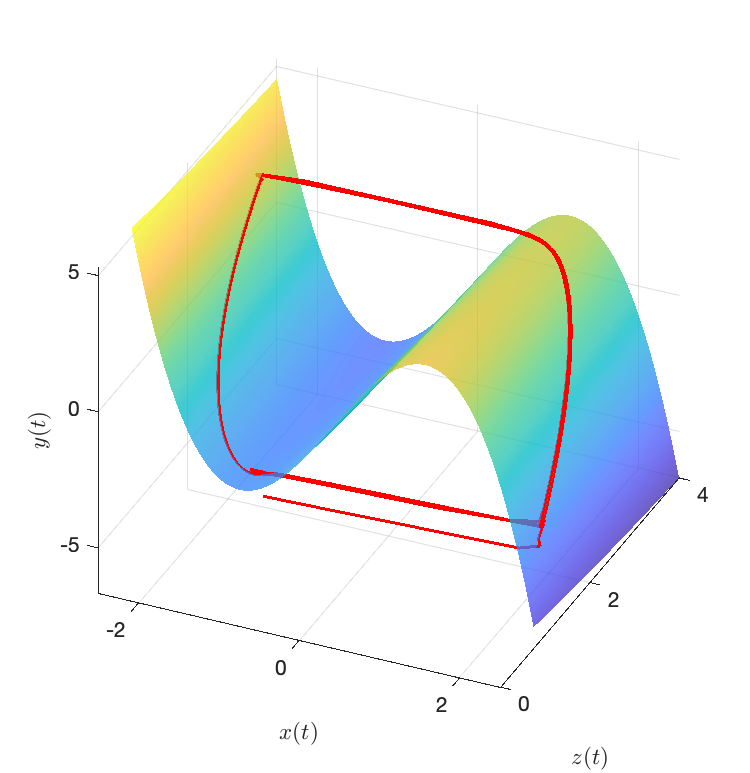}
\end{subfigure}
\hfill
\begin{subfigure}[ht]{0.45\textwidth}
\centering
\includegraphics[scale=0.30]{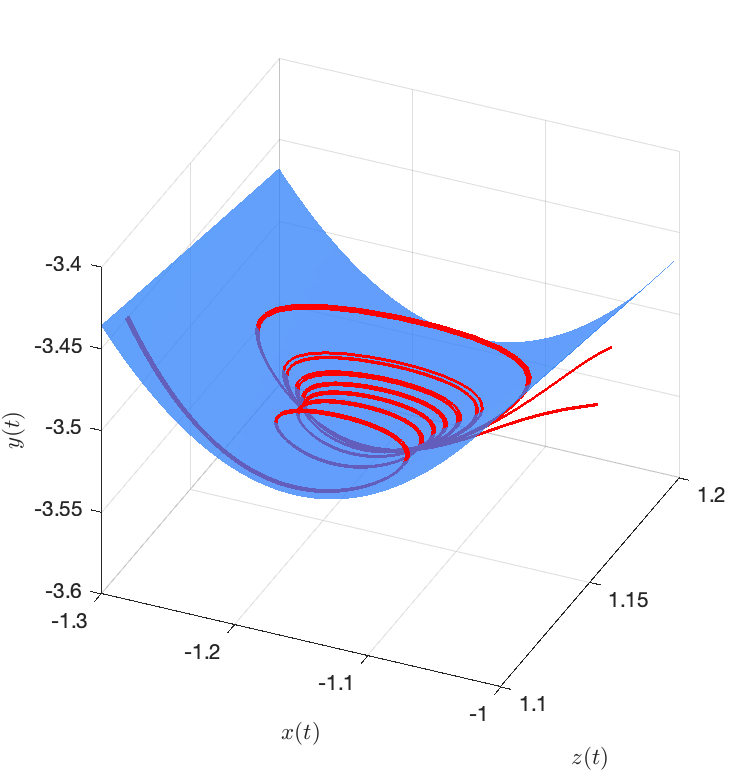}
\end{subfigure}
\caption{Unstable point, MMOs. Left: periodic orbit. Right: magnified view of the small oscillations of the orbit.}\label{Fig1c}
\end{subfigure}
\caption{Asymptotic behavior of the solution of model \eqref{ICC}.}
\end{figure}

\begin{figure}[ht]
\centering
\begin{subfigure}[ht]{0.3\textwidth}
\centering
\includegraphics[scale=0.28]{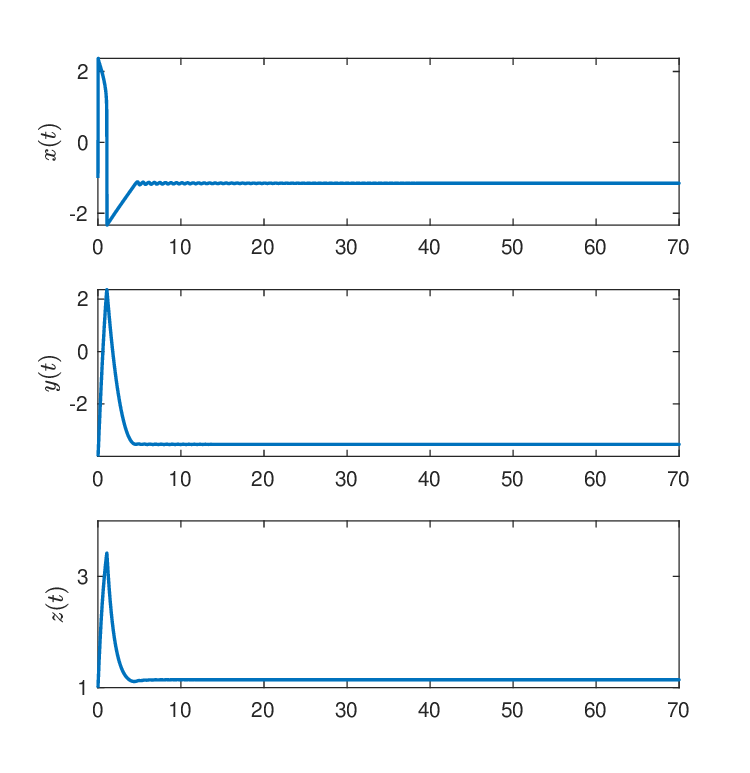}
\caption{Stable equilibrium ($\mu = 2.5$): no oscillations.}
\end{subfigure}
\hfill
\begin{subfigure}[ht]{0.3\textwidth}
\centering
\includegraphics[scale=0.28]{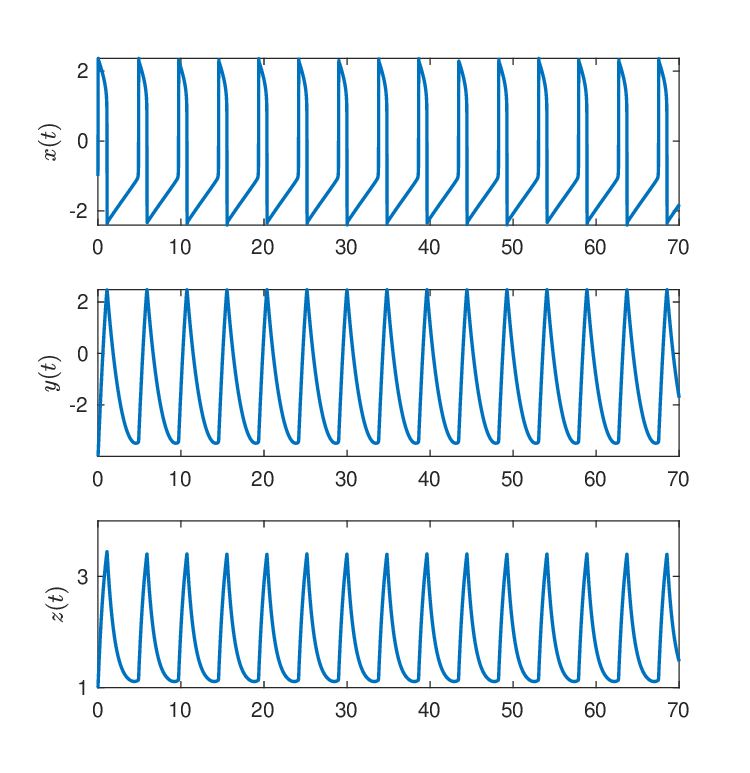}
\caption{Relaxation oscillations on a stable limit cycle ($\mu = 2.21$).}
\end{subfigure}
\hfill
\begin{subfigure}[ht]{0.3\textwidth}
\centering
\includegraphics[scale=0.28]{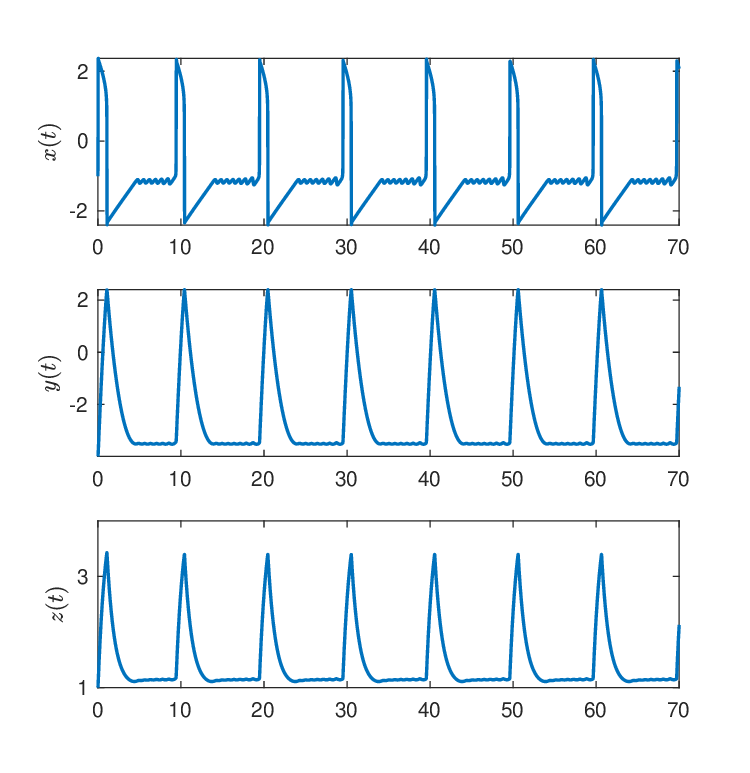}
\caption{Periodic Mixed-Mode Oscillations ($\mu = 2.4$).}
\end{subfigure}
\caption{Time evolution of $\left(x, y, z\right)^{\top}$ in the three deterministic regimes of model \eqref{ICC}.}\label{Fig1c2}
\end{figure}

We also display the time evolution of the variables $\left(x, y, z\right)^{\top}$ in the three regimes.

To conclude this section, we state the following inequality, which will be employed multiple times in the subsequent analysis:
\begin{equation}
\forall u, v \in \mathbb{R}, \forall \delta > 0, - \frac{1}{2\delta}u^{2} - \frac{\delta}{2}v^{2} \leq uv \leq \frac{1}{2\delta}u^{2} + \frac{\delta}{2}v^{2}.
\label{ine}
\end{equation}

We will later extend both the deterministic and stochastic formulations to a coupled two-cell (six-dimensional) ICC model.

\section{Stochastic Framework}\label{sec3}
In this section we introduce the stochastic framework used throughout the paper and formulate a stochastic version of the ICC model. We begin by recalling basic results on stochastic differential equations that will be required in the analysis. We then propose a stochastic perturbation of the deterministic ICC system and establish its well-posedness. Finally, we investigate the asymptotic behaviour of the stochastic solution in the case where the deterministic model possesses a stable equilibrium.

\subsection{Stochastic Setup}
Throughout this paper, we assume the usual conditions of the stochastic framework, see for example \cite{Kloeden1992}. Let $\left(\Omega, \mathcal{F}, \mathbb{P}\right)$ be a complete probability space with the natural filtration $\left\{\mathcal{F}_{t}\right\}_{t \geq 0}$. Consider a $d$-dimensional It{\^o} stochastic differential equation (SDE)
\begin{equation}
\left\{\begin{array}{ll}
\displaystyle \mathrm{d}U_{t} = f\left(t, U_{t}\right)\mathrm{d}t + g\left(t, U_{t}\right)\mathrm{d}W_{t}, & \displaystyle t \in \left[0, T\right],\\
\displaystyle U_{0} \text{ is given}, & \displaystyle
\end{array}\right.
\label{sde}
\end{equation}
\noindent where $T$ is a positive constant, the drift $f : \left[0, T\right] \times \mathbb{R}^{d} \to \mathbb{R}^{d}$ and the diffusion $g : \left[0, T\right] \times \mathbb{R}^{d} \to \mathbb{R}^{d \times m}$, and $U_{0}$ is the initial data. The process $W_{t}$ is a $m$-dimensional Wiener process adapted to $\left\{\mathcal{F}_{t}\right\}_{t \geq 0}$, and the initial data $U_{0}$ is independent of it, satisfying $\mathbb{E}\left[\left\|U_{0}\right\|^{2}\right] < \infty$, where $\left\|\cdot\right\|$ denotes the Euclidean norm for vectors and the Frobenius norm for matrices.

Let $\mathcal{M}^{2}\left(\left[0, T\right]; \mathbb{R}^{d}\right)$ denote the family of all $d$-valued measurable $\left\{\mathcal{F}_{t}\right\}$-adapted processes $X_{t}$ such that
\begin{equation*}
\mathbb{E}\left[\int_{0}^{T}\left\|X_{s}\right\|^{2}\mathrm{d}s\right] < \infty.
\end{equation*}

We now present an existence and uniqueness result for the problem \eqref{sde} under some regularity hypotheses on the data:

\begin{theorem}\label{eu}
Suppose that

\item[1. Measurability.] The functions $f$ and $g$ are jointly measurable in $\left(t, x\right)$.

\item[2. Local Lipschitz continuity.] For every integer $n \geq 1$, there exists $L_{n} > 0$ such that, for all $x, y \in \mathbb{R}^{d}$ and $t \in \left[0, T\right]$, with $\left\|x\right\| \leq n$ and $\left\|y\right\| \leq n$,
\begin{equation*}
\left\|f\left(t, x\right) - f\left(t, y\right)\right\| + \left\|g\left(t, x\right) - g\left(t, y\right)\right\| \leq L_{n}\left\|x - y\right\|.
\end{equation*}

\item[3. Growth bound.] There exists $K > 0$ such that, for all $x \in \mathbb{R}^{d}$ and $t \in \left[0, T\right]$,
\begin{equation*}
x^{\top}f\left(t, x\right) + \frac{1}{2}\left\|g\left(t, x\right)\right\|^{2} \leq K\left(1 + \left\|x\right\|^{2}\right).
\end{equation*}

Then, the SDE \eqref{sde} admits a unique solution $U_{t} \in \mathcal{M}^{2}\left(\left[0, T\right]; \mathbb{R}^{d}\right)$.
\end{theorem}

The proof of this result can be found in \cite{Mao2007}.

We now state It{\^o}'s formula for a scalar function. Let $V \in \mathcal{C}^{1, 2}$ be a function which is continuously differentiable with respect to the first variable and twice continuously differentiable with respect to the second one. Then, the function $V\left(t, U_{t}\right)$, where $U_{t}$ is the solution of \eqref{sde}, satisfies the following SDE:
\begin{equation}
\mathrm{d}V\left(t, U_{t}\right) = LV\left(t, U_{t}\right)\mathrm{d}t + \nabla V^{\top}g\left(t, U_{t}\right)\mathrm{d}W_{t},
\label{Ito}
\end{equation}
\noindent where we define the operator $L$ acting on $V$ as
\begin{equation*}
LV = \frac{\partial V}{\partial t} + \nabla V^{\top}f + \frac{1}{2}\Tr\left(g^{\top}Hg\right),
\end{equation*}
\noindent where $\nabla V$ is the gradient of $V$ with respect to the second variable, $H$ is its Hessian matrix with respect to the second variable and $f$ and $g$ are the functions of \eqref{sde}.

\subsection{Stochastic ICC model}
In this subsection we introduce a stochastic version of the ICC model and establish its main theoretical properties. We first show that the third variable remains positive, consistently with its interpretation as a concentration. We then prove existence and uniqueness of solutions to the stochastic system. Finally, we analyse the asymptotic behaviour of the stochastic solution in the case where the deterministic model admits a unique stable equilibrium.

To make model \eqref{ICC} more realistic, it is natural to incorporate randomness in the parameters corresponding to physiological properties of the cells in their respective equations. Since the first variable represents the cell electrical activity and the second describes recovery dynamics, it is natural to incorporate randomness into their respective equations. In contrast, the third variable is mainly driven by the first one and models a concentration process; therefore, no diffusion term is included in its equation.

Let $\xi_{t} = \left(X_{t}, Y_{t}, Z_{t}\right)^{\top}$. We consider the following stochastic differential equation:
\begin{equation}
\left\{\begin{array}{ll}
\displaystyle \mathrm{d}X_{t} = \tau\left(b_{0}Y_{t} + b_{1}X_{t} - X_{t}^{3} - \phi_{f}\left(Z_{t}\right)\right)\mathrm{d}t + g_{1}\left(\xi_{t}\right)\mathrm{d}W_{t}^{1} & \displaystyle = f_{1}\left(\xi_{t}\right)\mathrm{d}t + g_{1}\left(\xi_{t}\right)\mathrm{d}W_{t}^{1},\\
\displaystyle \mathrm{d}Y_{t} = \tau\varepsilon k\left(a_{0}X_{t} + a_{1}Y_{t} + a_{2}\right)\mathrm{d}t + \sqrt{\varepsilon}g_{2}\left(\xi_{t}\right)\mathrm{d}W_{t}^{2} & \displaystyle = f_{2}\left(\xi_{t}\right)\mathrm{d}t + \sqrt{\varepsilon}g_{2}\left(\xi_{t}\right)\mathrm{d}W_{t}^{2},\\
\displaystyle \mathrm{d}Z_{t} = \tau\varepsilon\left(\phi_{r}\left(X_{t}\right) - \frac{Z_{t} - z_{b}}{\tau_{z}}\right)\mathrm{d}t & \displaystyle = f_{3}\left(\xi_{t}\right)\mathrm{d}t,\\
\displaystyle \xi_{0} \text{ is given}, & \displaystyle
\end{array}\right.
\label{SICC}
\end{equation}
\noindent where $W_{t}^{1}$ and $W_{t}^{2}$ are independent scalar Wiener processes, and the functions $g_{1}$ and $g_{2}$ describe the intensity of the stochastic perturbations and are assumed to satisfy suitable regularity conditions.

Since the slow drift term moves the variable $Y_{t}$ by order $\varepsilon\Delta t$ over a small time interval $\Delta t$, the variance of the stochastic perturbation acting on $Y_{t}$ is chosen to scale as $\varepsilon\Delta t$. This explains the factor $\sqrt{\varepsilon}$ in the diffusion term. Consequently, system \eqref{SICC} preserves the slow-fast structure of the deterministic model.

The proposed model arises from perturbing the ICC system directly, but it can also be formally obtained by introducing white-noise perturbations into the linear parameters of the first two equations. In particular, taking $b_{1}$ as $b_{1} + \frac{\sigma_{1}}{\tau}\frac{\mathrm{d}W_{t}^{1}}{\mathrm{d}t}$, $b_{0}$ as $b_{0} + \frac{\sigma_{2}}{\tau}\frac{\mathrm{d}W_{t}^{1}}{\mathrm{d}t}$, $\varepsilon a_{0}$ as $\varepsilon a_{0} + \sqrt{\varepsilon}\frac{\sigma_{3}}{\tau k}\frac{\mathrm{d}W_{t}^{2}}{\mathrm{d}t}$ and $\varepsilon a_{1}$ as $\varepsilon a_{1} + \sqrt{\varepsilon}\frac{\sigma_{4}}{\tau k}\frac{\mathrm{d}W_{t}^{2}}{\mathrm{d}t}$ formally leads to the following stochastic system:
\begin{equation}
\left\{\begin{array}{l}
\displaystyle \mathrm{d}X_{t} = \tau\left(b_{0}Y_{t} + b_{1}X_{t} - X_{t}^{3} - \phi_{f}\left(Z_{t}\right)\right)\mathrm{d}t + \left(\sigma_{1}X_{t} + \sigma_{2}Y_{t}\right)\mathrm{d}W_{t}^{1},\\
\displaystyle \mathrm{d}Y_{t} = \tau\varepsilon k\left(a_{0}X_{t} + a_{1}Y_{t} + a_{2}\right)\mathrm{d}t + \sqrt{\varepsilon}\left(\sigma_{3}X_{t} + \sigma_{4}Y_{t}\right)\mathrm{d}W_{t}^{2},\\
\displaystyle \mathrm{d}Z_{t} = \tau\varepsilon\left(\phi_{r}\left(X_{t}\right) - \frac{Z_{t} - z_{b}}{\tau_{z}}\right)\mathrm{d}t,\\
\displaystyle \xi_{0} \text{ is given}.
\end{array}\right.
\label{lineal}
\end{equation}

As a first theoretical result, we show that the stochastic model \eqref{SICC} preserves the positivity of the slow variable $Z_{t}$ whenever the initial condition is positive.

\begin{theorem}[Positivity of the variable $Z_{t}$]\label{posi}
Let $Z_{0} \in \mathbb{R}_{+}$. Then, the third component $Z_{t}$ associated with model \eqref{SICC} satisfies $Z_{t} \in \mathbb{R}_{+}$ for all $t \in \left[0, T\right]$.
\end{theorem}

\begin{proof}
\noindent Let us consider the differential equation of the variable $Z_{t}$:
\begin{equation*}
\mathrm{d}Z_{t} = \tau\varepsilon\left(\phi_{r}\left(X_{t}\right) - \frac{Z_{t} - z_{b}}{\tau_{z}}\right)\mathrm{d}t,
\end{equation*}
\noindent which is deterministic.

\noindent Solving the differential equation, we obtain
\begin{equation}
Z_{t} = \left(\int_{0}^{t}\tau\varepsilon\left(\phi_{r}\left(X_{s}\right) + \frac{z_{b}}{\tau_{z}}\right)e^{\frac{\tau\varepsilon}{\tau_{z}}s}\mathrm{d}s + Z_{0}\right)e^{-\frac{\tau\varepsilon}{\tau_{z}}t}.
\label{Zt}
\end{equation}

\noindent Note that for all $x \in \mathbb{R}$
\begin{equation*}
\phi_{r}\left(x\right) = \frac{\lambda}{1 + \exp\left(- \rho\left(x - x_{on}\right)\right)} > 0.
\end{equation*}

\noindent Consequently, everything on the right-hand side of \eqref{Zt} is positive, thus concluding the proof.
\end{proof}

We now state an existence and uniqueness result for the stochastic model. In what follows, we denote $\mathcal{R} \coloneqq \mathbb{R}^{2} \times \mathbb{R}_{+}$.

\begin{theorem}[Existence and uniqueness]\label{EUICC}
Assume that the functions $g_{1}$ and $g_{2}$ are locally Lipschitz and satisfy the global linear growth condition
\begin{equation*}
\frac{1}{2}g_{1}\left(\xi\right)^{2} + \frac{1}{2}\varepsilon g_{2}\left(\xi\right)^{2} \leq K\left(1 + \left\|\xi\right\|^{2}\right),
\end{equation*}
\noindent for all $\xi \in \mathcal{R}$, for some constant $K > 0$.

Then, for any given initial value $\xi_{0} \in \mathcal{R}$, the stochastic ICC model \eqref{SICC} admits a unique strong solution $\xi_{t} \in \mathcal{M}^{2}\left(\left[0, T\right]; \mathcal{R}\right)$ for all $t \in \left[0, T\right]$.
\end{theorem}

\begin{proof}
\noindent In order to apply Theorem~\ref{eu}, we first verify that its assumptions are satisfied.

\noindent The local Lipschitz condition holds, since all the terms in the drift are polynomial functions, except for $\phi_{r}\left(x\right) = \frac{\lambda}{1 + \exp\left(- \rho\left(x - x_{on}\right)\right)}$, which is continuous and differentiable with bounded derivative, and $\phi_{f}\left(z\right) = \frac{\mu z}{z + z_{0}}$, which is continuous and differentiable with bounded derivative in our domain, that is, for $z > 0$. Moreover, by hypothesis, the diffusion coefficients $g_{1}$ and $g_{2}$ are locally Lipschitz with respect to the state variable.

\noindent On the other hand, we study the monotone condition. Let
\begin{equation*}
\xi_{t}^{\top}\begin{pmatrix}
\displaystyle f_{1}\left(\xi_{t}\right)\\
\displaystyle f_{2}\left(\xi_{t}\right)\\
\displaystyle f_{3}\left(\xi_{t}\right)
\end{pmatrix} + \frac{1}{2}g_{1}\left(\xi_{t}\right)^{2} + \frac{1}{2}\varepsilon g_{2}\left(\xi_{t}\right)^{2}.
\end{equation*}

\noindent We study the first product:
\begin{equation*}
\xi_{t}^{\top}\begin{pmatrix}
\displaystyle f_{1}\left(\xi_{t}\right)\\
\displaystyle f_{2}\left(\xi_{t}\right)\\
\displaystyle f_{3}\left(\xi_{t}\right)
\end{pmatrix} \leq \underbrace{A_{1}Y_{t} + A_{2}Z_{t}}_{I} + \underbrace{B_{1}X_{t}^{2} + B_{2}Y_{t}^{2}}_{II} + \underbrace{C_{1}X_{t}Y_{t} + C_{2}X_{t}\phi_{f}\left(Z_{t}\right)}_{III} + \underbrace{\left(- \frac{\tau\varepsilon}{\tau_{z}}Z_{t}^{2} - \tau X_{t}^{4}\right)}_{IV},
\end{equation*}
\noindent where

\begin{itemize}
\item $A_{1} = \tau\varepsilon ka_{2}$.

\item $A_{2} = \tau\varepsilon\left(\phi_{r}\left(X_{t}\right) + \frac{z_{b}}{\tau_{z}}\right)$.

\item $B_{1} = \tau b_{1}$.

\item $B_{2} = - \tau\varepsilon ka_{1}$.

\item $C_{1} = \tau\left(\left|b_{0}\right| + \varepsilon ka_{0}\right)$.

\item $C_{2} = - \tau$.
\end{itemize}

\noindent We study term by term. To $I$, we apply the right inequality of \eqref{ine} and use the boundedness of the function $\phi_{r}\left(x\right) = \frac{\lambda}{1 + \exp\left(- \rho\left(x - x_{on}\right)\right)} \leq \lambda$, for all $x \in \mathbb{R}$:
\begin{equation*}
I \leq \frac{A_{1}}{2} + \frac{A_{1}}{2}Y_{t}^{2} + \frac{\tau\varepsilon\left(\lambda + \frac{z_{b}}{\tau_{z}}\right)}{2} + \frac{\tau\varepsilon\left(\lambda + \frac{z_{b}}{\tau_{z}}\right)}{2}Z_{t}^{2}.
\end{equation*}

\noindent Note that $II$ is already in the required form. Let us study $III$. Note that $C_{1} > 0$. Then, we use again the right inequality of \eqref{ine} to obtain the following:
\begin{equation*}
C_{1}X_{t}Y_{t} \leq \frac{C_{1}}{2}X_{t}^{2} + \frac{C_{1}}{2}Y_{t}^{2}.
\end{equation*}

\noindent On the other hand, since $C_{2} < 0$, we use again the left inequality of \eqref{ine}, along with the bound $\phi_{f}\left(z\right) = \frac{\mu z}{z + z_{0}} \leq \mu$, for all $z > 0$, to obtain
\begin{equation*}
C_{2}X_{t}\phi_{f}\left(Z_{t}\right) \leq - \frac{C_{2}}{2}X_{t}^{2} - \frac{C_{2}}{2}\phi_{f}\left(Z_{t}\right)^{2} \leq - \frac{C_{2}}{2}X_{t}^{2} - \frac{C_{2}}{2}\mu^{2}.
\end{equation*}

\noindent Finally, we use the fact that $IV < 0$ to obtain
\begin{equation*}
\xi_{t}^{\top}\begin{pmatrix}
\displaystyle f_{1}\left(\xi_{t}\right)\\
\displaystyle f_{2}\left(\xi_{t}\right)\\
\displaystyle f_{3}\left(\xi_{t}\right)
\end{pmatrix} \leq \frac{A_{1}}{2} + \frac{\tau\varepsilon\left(\lambda + \frac{z_{b}}{\tau_{z}}\right)}{2} - \frac{C_{2}}{2}\mu^{2} + \left(\frac{C_{1}}{2} - \frac{C_{2}}{2}\right)X_{t}^{2} + \left(\frac{A_{1}}{2} + \frac{C_{1}}{2}\right)Y_{t}^{2} + \frac{\tau\varepsilon\left(\lambda + \frac{z_{b}}{\tau_{z}}\right)}{2}Z_{t}^{2}.
\end{equation*}

\noindent Using the global linear growth bound on the hypothesis,
\begin{equation*}
\xi_{t}^{\top}\begin{pmatrix}
\displaystyle f_{1}\left(\xi_{t}\right)\\
\displaystyle f_{2}\left(\xi_{t}\right)\\
\displaystyle f_{3}\left(\xi_{t}\right)
\end{pmatrix} + \frac{1}{2}g_{1}\left(\xi_{t}\right)^{2} + \frac{1}{2}\varepsilon g_{2}\left(\xi_{t}\right)^{2} \leq K^{\ast}\left(1 + \left\|\xi_{t}\right\|^{2}\right),
\end{equation*}
\noindent for some constant $K^{\ast} > 0$, which yields the monotone condition. Then, by Theorem~\ref{eu}, we conclude that there exists a unique solution $\xi_{t} \in \mathcal{M}^{2}\left(\left[0, T\right]; \mathcal{R}\right)$ to problem \eqref{SICC}.
\end{proof}

\begin{remark}
Existence and uniqueness also hold for the linear stochastic system \eqref{lineal}. Indeed, its drift and diffusion coefficients are locally Lipschitz, and the diffusion terms satisfy a global linear growth condition. In particular, one may take
\begin{equation*}
K = \max\left\{\sigma_{1}^{2} + \varepsilon\sigma_{3}^{2}, \sigma_{2}^{2} + \varepsilon\sigma_{4}^{2}\right\},
\end{equation*}
so that the assumptions of Theorem~\ref{EUICC} are fulfilled.
\end{remark}

We now investigate the long-time behaviour of the stochastic solution $\left(X_{t}, Y_{t}, Z_{t}\right)^{\top}$ in the case where the deterministic system admits a unique globally asymptotically stable equilibrium $\left(x^{\ast}, y^{\ast}, z^{\ast}\right)^{\top}$.

In the deterministic setting, global asymptotic stability implies convergence of all trajectories towards the equilibrium. In the stochastic framework, however, the deterministic equilibrium is generally not an equilibrium of the perturbed system. Therefore, instead of convergence to a fixed point, one seeks estimates describing the behaviour of the stochastic solution relative to the deterministic equilibrium.

The following result provides a global estimate of the average oscillations of the stochastic trajectories around the deterministic equilibrium and is one of the main results of the paper. It can be interpreted as a quantitative measure of stability in mean, and follows an approach similar to that developed in \cite{Berrhazi2018, Caraballo2018}.

\begin{theorem}\label{th}
Assume that \eqref{ICC} has a unique globally asymptotically stable equilibrium $\left(x^{\ast}, y^{\ast}, z^{\ast}\right)^{\top}$ and functions $g_{1}$ and $g_{2}$ of the stochastic model \eqref{SICC} satisfy the following global linear growth bounds:
\begin{equation*}
\frac{1}{2}g_{1}\left(\xi\right)^{2} + \frac{1}{2}\varepsilon g_{2}\left(\xi\right)^{2} \leq K\left(1 + \left\|\xi\right\|^{2}\right),
\end{equation*}
\noindent for all $\xi \in \mathcal{R}$, with
\begin{equation*}
K < \min\left\{\frac{\tau b_{0}a_{1}}{2a_{0}\beta}, \frac{\tau\varepsilon}{2\tau_{z}\beta}\right\},
\end{equation*}
\noindent where $\beta \coloneqq \max\left\{1, - \frac{b_{0}}{\varepsilon ka_{0}}\right\} > 0$. Then, for any given initial condition $\xi_{0} \in \mathcal{R}$, there exists a positive constant $C$ such that the solution $\xi_{t}$ of model \eqref{SICC} satisfies
\begin{equation}
\limsup_{t \to \infty}\frac{1}{t}\mathbb{E}\left[\int_{0}^{t}\left(\left(X_{s} - x^{\ast}\right)^{2} + \left(Y_{s} - y^{\ast}\right)^{2} + \left(Z_{s} - z^{\ast}\right)^{2}\right)\mathrm{d}s\right] \leq C.
\label{condi}
\end{equation}
\end{theorem}

\begin{proof}
\noindent We define the following Lyapunov function for each $t \in \left[0, T\right]$:
\begin{equation*}
V\left(\xi_{t}\right) = \frac{1}{2}\left(X_{t} - x^{\ast}\right)^{2} + \frac{\alpha}{2}\left(Y_{t} - y^{\ast}\right)^{2} + \frac{1}{2}\left(Z_{t} - z^{\ast}\right)^{2},
\end{equation*}
\noindent with $\alpha$ a positive constant to be specified later.

\noindent Hereinafter we denote $\epsilon_{x, t} = X_{t} - x^{\ast}$, $\epsilon_{y, t} = Y_{t} - y^{\ast}$ and $\epsilon_{z, t} = Z_{t} - z^{\ast}$.

\noindent Using the It{\^o}'s formula \eqref{Ito}, we obtain
\begin{equation}
\mathrm{d}V = LV\mathrm{d}t + \epsilon_{x, t}g_{1}\left(\xi_{t}\right)\mathrm{d}W_{t}^{1} + \alpha\epsilon_{y, t}\sqrt{\varepsilon}g_{2}\left(\xi_{t}\right)\mathrm{d}W_{t}^{2},
\label{dV}
\end{equation}
\noindent where
\begin{equation*}
LV\left(\xi_{t}\right) = \epsilon_{x, t}f_{1}\left(\xi_{t}\right) + \alpha\epsilon_{y, t}f_{2}\left(\xi_{t}\right) + \epsilon_{z, t}f_{3}\left(\xi_{t}\right) + \frac{1}{2}g_{1}\left(\xi_{t}\right)^{2} + \frac{1}{2}\alpha\varepsilon g_{2}\left(\xi_{t}\right)^{2}.
\end{equation*}

\noindent If $\left(x^{\ast}, y^{\ast}, z^{\ast}\right)^{\top}$ is a stable equilibrium point, then $f_{i}\left(\left(x^{\ast}, y^{\ast}, z^{\ast}\right)^{\top}\right) = 0$, for $i = 1, 2, 3$, from which we can write that:
\begin{equation*}
\begin{array}{rl}
\displaystyle 0 = & \displaystyle - b_{0}y^{\ast} - b_{1}x^{\ast} + \left(x^{\ast}\right)^{3} + \frac{\mu z^{\ast}}{z^{\ast} + z_{0}},\\
\displaystyle a_{2} = & \displaystyle - a_{0}x^{\ast} - a_{1}y^{\ast},\\
\displaystyle \frac{z_{b}}{\tau_{z}} = & \displaystyle - \frac{\lambda}{1 + \exp\left(- \rho\left(x^{\ast} - x_{on}\right)\right)} + \frac{z^{\ast}}{\tau_{z}}.
\end{array}
\end{equation*}

\noindent Substituting these values into the expressions of the functions $f_{1}$, $f_{2}$ and $f_{3}$, we have:
\begin{equation}
LV\left(\xi_{t}\right) = A\left(X_{t}\right)\epsilon_{x, t}^{2} + B\epsilon_{y, t}^{2} + C\epsilon_{z, t}^{2} + D\epsilon_{x, t}\epsilon_{y, t} + E\left(Z_{t}\right)\epsilon_{x, t} + F\left(X_{t}\right)\epsilon_{z, t} + G\left(\xi_{t}\right),
\label{LV1}
\end{equation}
\noindent where, using $\epsilon_{x, t}\left(- X_{t}^{3} + \left(x^{\ast}\right)^{3}\right) = \epsilon_{x, t}^{2}\left(- X_{t}^{2} - X_{t}x^{\ast} - \left(x^{\ast}\right)^{2}\right)$,

\begin{itemize}
\item $A\left(X_{t}\right) = \tau\left(b_{1} - X_{t}^{2} - X_{t}x^{\ast} - \left(x^{\ast}\right)^{2}\right)$.

\item $B = \alpha\tau\varepsilon ka_{1}$.

\item $C = - \frac{\tau\varepsilon}{\tau_{z}}$.

\item $D = \tau\left(b_{0} + \alpha\varepsilon ka_{0}\right)$.

\item $E\left(Z_{t}\right) = - \tau\left(\phi_{f}\left(Z_{t}\right) - \phi_{f}\left(z^{\ast}\right)\right)$.

\item $F\left(X_{t}\right) = \tau\varepsilon\left(\phi_{r}\left(X_{t}\right) - \phi_{r}\left(x^{\ast}\right)\right)$.

\item $G\left(\xi_{t}\right) = \frac{1}{2}g_{1}\left(\xi_{t}\right)^{2} + \frac{1}{2}\alpha\varepsilon g_{2}\left(\xi_{t}\right)^{2}$.
\end{itemize}

\noindent Choosing $\alpha = - \frac{b_{0}}{\varepsilon ka_{0}} > 0$, then $D = 0$. Also, to bound some summands, we will use the inequalities \eqref{ine}.

\begin{itemize}
\item $E\left(Z_{t}\right)\epsilon_{x, t}$: Since $\phi_{f}$ is continuously differentiable on $\mathbb{R}_{+}$, the mean value theorem can be applied: there exist $\zeta_{z}$ between $z^{\ast}$ and $Z_{t}$ such that $E\left(Z_{t}\right) = - \tau\phi_{f}'\left(\zeta_{z}\right)\epsilon_{z, t}$, with $\phi_{f}'\left(z\right) = \frac{\mu z_{0}}{\left(z + z_{0}\right)^{2}}$. Moreover, $0 < \phi_{f}'\left(z\right) \leq \frac{\mu}{z_{0}}$, $\forall z \in \mathbb{R}_{+}$. Let $L_{1} = \frac{\tau\mu}{z_{0}} > 0$, we use \eqref{ine}: $\forall \delta_{1} > 0$, $\epsilon_{x, t}\epsilon_{z, t} \geq - \frac{1}{2\delta_{1}}\epsilon_{x, t}^{2} - \frac{\delta_{1}}{2}\epsilon_{z, t}^{2}$. Taking into account that $- \tau\phi_{f}'\left(\zeta_{z}\right)$ is negative, we get
\begin{equation}
- \tau\phi_{f}'\left(\zeta_{z}\right)\epsilon_{x, t}\epsilon_{z, t} \leq \frac{\tau\phi_{f}'\left(\zeta_{z}\right)}{2\delta_{1}}\epsilon_{x, t}^{2} + \frac{\tau\phi_{f}'\left(\zeta_{z}\right)\delta_{1}}{2}\epsilon_{z, t}^{2} \leq \frac{L_{1}}{2\delta_{1}}\epsilon_{x, t}^{2} + \frac{L_{1}\delta_{1}}{2}\epsilon_{z, t}^{2}.
\label{eqE}
\end{equation}

\item $F\left(X_{t}\right)\epsilon_{z, t}$: Since $\phi_{r}$ is continuously differentiable on $\mathbb{R}$, the mean value theorem can be applied: there exist $\zeta_{x}$ between $x^{\ast}$ and $X_{t}$ such that $F\left(X_{t}\right) = \tau\varepsilon\phi_{r}'\left(\zeta_{x}\right)\epsilon_{x, t}$, with $\phi_{r}'\left(x\right) = \frac{\lambda\rho\exp\left(- \rho\left(x - x_{on}\right)\right)}{\left(1 + \exp\left(- \rho\left(x - x_{on}\right)\right)\right)^{2}}$. Moreover, $0 < \phi_{r}'\left(x\right) \leq \frac{\lambda\rho}{4}$, $\forall x \in \mathbb{R}$. Let $L_{2} = \frac{\tau\varepsilon\lambda\rho}{4} > 0$, we use \eqref{ine}: $\forall \delta_{2} > 0$, $\epsilon_{x, t}\epsilon_{z, t} \leq \frac{1}{2\delta_{2}}\epsilon_{x, t}^{2} + \frac{\delta_{2}}{2}\epsilon_{z, t}^{2}$. Taking into account that $\tau\varepsilon\phi_{r}'\left(\zeta_{x}\right)$ is positive, we get
\begin{equation}
\tau\varepsilon\phi_{r}'\left(\zeta_{x}\right)\epsilon_{x, t}\epsilon_{z, t} \leq \frac{\tau\varepsilon\phi_{r}'\left(\zeta_{x}\right)}{2\delta_{2}}\epsilon_{x, t}^{2} + \frac{\tau\varepsilon\phi_{r}'\left(\zeta_{x}\right)\delta_{2}}{2}\epsilon_{z, t}^{2} \leq \frac{L_{2}}{2\delta_{2}}\epsilon_{x, t}^{2} + \frac{L_{2}\delta_{2}}{2}\epsilon_{z, t}^{2}.
\label{eqF}
\end{equation}

\item $G\left(\xi_{t}\right)$: For this term, we use the global linear growth bounds and the inequality $\forall a, b \in \mathbb{R}$, $a^{2} \leq 2\left(a - b\right)^{2} + 2b^{2}$. Taking $\beta \coloneqq \max\left\{1, \alpha\right\} > 0$, we obtain
\begin{equation}
G\left(\xi_{t}\right) \leq \beta K\left(1 + \left\|\xi_{t}\right\|^{2}\right) \leq 2\beta K\left(1 + \epsilon_{x, t}^{2} + \epsilon_{y, t}^{2} + \epsilon_{z, t}^{2} + \left\|\left(x^{\ast}, y^{\ast}, z^{\ast}\right)\right\|^{2}\right).
\label{eqG}
\end{equation}
\end{itemize}

\noindent Inserting the bounds \eqref{eqE}--\eqref{eqG} into the equation \eqref{LV1}, we obtain
\begin{equation}
LV \leq \tilde{A}\left(X_{t}\right)\epsilon_{x, t}^{2} + \tilde{B}\epsilon_{y, t}^{2} + \tilde{C}\epsilon_{z, t}^{2} + \tilde{D},
\label{LV2}
\end{equation}
\noindent where now:

\begin{itemize}
\item $\tilde{A}\left(X_{t}\right) = A\left(X_{t}\right) + \frac{L_{1}}{2\delta_{1}} + \frac{L_{2}}{2\delta_{2}} + 2\beta K$.

\item $\tilde{B} = B + 2\beta K$.

\item $\tilde{C} = C + \frac{L_{1}\delta_{1}}{2} + \frac{L_{2}\delta_{2}}{2} + 2\beta K$.

\item $\tilde{D} = 2\beta K + 2\beta K\left(x^{\ast}\right)^{2} + 2\beta K\left(y^{\ast}\right)^{2} + 2\beta K\left(z^{\ast}\right)^{2}$.
\end{itemize}

\noindent We now study each term separately. Firstly, note that $\tilde{D}$ is always positive and only depends on $K$ and the equilibrium. On the other hand, $\tilde{B} < 0$ because $\tilde{B} = B + 2\beta K = - \frac{\tau b_{0}a_{1}}{a_{0}} + 2\beta K$ is always negative as $K < \frac{\tau b_{0}a_{1}}{2a_{0}\beta}$.

\noindent Also, we choose $\delta_{1}$ and $\delta_{2}$ such that $\tilde{C} < 0$, i.e., we choose for example $0 < \delta_{1} < \frac{2}{L_{1}}\left(- C - \frac{L_{2}\delta_{2}}{2} - 2\beta K\right)$ and $0 < \delta_{2} < \frac{2}{L_{2}}\left(- C - 2\beta K\right)$. Note that, $- C - 2\beta K = \frac{\tau\varepsilon}{\tau_{z}} - 2\beta K$ is always positive as $K < \frac{\tau\varepsilon}{2\tau_{z}\beta}$.

\noindent Finally, we study the parabola $\tilde{A}\left(X_{t}\right)$. Let $\tilde{\delta}$ a positive value chosen to prevent division by zero and let $L_{3} = \max\left\{\tilde{A}\left(X_{t}\right), \tilde{\delta}\mathbb{I}_{\left\{\max\left\{\tilde{A}\left(X_{t}\right)\right\} = 0\right\}}\right\} \geq 0$, where $\mathbb{I}$ is the indicator function. Note that if the parabola is always negative, $L_{3} = 0$ and otherwise, $L_{3}$ is the value of the parabola in its vertex or $\tilde{\delta}$ if the parabola in its vertex is equal to zero. Let $x^{-}$ and $x^{+}$ be the roots (or root) of the parabola (if they exist), with $x^{-} \leq x^{+}$, we choose $\delta_{3}$ a positive value, such that $\tilde{A}\left(x^{-} - \delta_{3}\right) = \tilde{A}\left(x^{+} + \delta_{3}\right) = - L_{3}$. Then, we split the parabola as follows:
\begin{equation*}
\tilde{A}\left(X_{t}\right) = \tilde{A}\left(X_{t}\right)^{+} + \tilde{A}\left(X_{t}\right)^{-},
\end{equation*}
\noindent where
\begin{equation*}
\tilde{A}\left(X_{t}\right)^{+} \coloneqq \left\{\begin{array}{ll}
\displaystyle 0, & \displaystyle X_{t} \notin \left[x^{-} - \delta_{3}, x^{+} + \delta_{3}\right],\\
\displaystyle \tilde{A}\left(X_{t}\right) + L_{3}, & \displaystyle X_{t} \in \left[x^{-} - \delta_{3}, x^{+} + \delta_{3}\right],
\end{array}\right.
\end{equation*}
\begin{equation*}
\tilde{A}\left(X_{t}\right)^{-} \coloneqq \left\{\begin{array}{ll}
\displaystyle \tilde{A}\left(X_{t}\right), & \displaystyle X_{t} \notin \left[x^{-} - \delta_{3}, x^{+} + \delta_{3}\right],\\
\displaystyle - L_{3}, & \displaystyle X_{t} \in \left[x^{-} - \delta_{3}, x^{+} + \delta_{3}\right].
\end{array}\right.
\end{equation*}

\noindent Note that $\tilde{A}\left(X_{t}\right)^{-}$ is always negative. In particular, $\tilde{A}\left(X_{t}\right)^{-} \leq - L_{3} < 0$ if the parabola is not necessarily negative. Also, $\tilde{A}\left(X_{t}\right)^{+}$ is always nonnegative and is bounded by $2L_{3}$. Therefore,
\begin{equation*}
\tilde{A}\left(X_{t}\right)^{+}\epsilon_{x, t}^{2} \leq 2L_{3}\max_{X_{t} \in \left[x^{-} - \delta_{3}, x^{+} + \delta_{3}\right]}\left\{\epsilon_{x, t}^{2}\right\} \eqqcolon \tilde{E} < \infty.
\end{equation*}

\noindent This leads us to the following inequality from \eqref{LV2}:
\begin{equation*}
LV \leq C_{2}\left(\epsilon_{x, t}^{2} + \epsilon_{y, t}^{2} + \epsilon_{z, t}^{2}\right) + C_{1},
\end{equation*}
\noindent where,

\begin{itemize}
\item $C_{1} = \tilde{D} + \tilde{E} \geq 0$.

\item $C_{2} = \max{\left\{\max{\left\{\tilde{A}\left(X_{t}\right)^{-}\right\}}, \tilde{B}, \tilde{C}\right\}} < 0$.
\end{itemize}

\noindent Substituting the above inequality into \eqref{dV},
\begin{equation*}
\mathrm{d}V \leq \left(C_{2}\left(\epsilon_{x, t}^{2} + \epsilon_{y, t}^{2} + \epsilon_{z, t}^{2}\right) + C_{1}\right)\mathrm{d}t + \epsilon_{x, t}g_{1}\left(\xi_{t}\right)\mathrm{d}W_{t}^{1} + \alpha\epsilon_{y, t}\sqrt{\varepsilon}g_{2}\left(\xi_{t}\right)\mathrm{d}W_{t}^{2}.
\end{equation*}

\noindent Integrating both sides between 0 and $t$ and taking expectation, we get:
\begin{equation*}
0 \leq \mathbb{E}\left[V\left(\xi_{t}\right)\right] \leq V\left(\xi_{0}\right) + C_{2}\mathbb{E}\left[\int_{0}^{t}\left(\epsilon_{x, s}^{2} + \epsilon_{y, s}^{2} + \epsilon_{z, s}^{2}\right)\mathrm{d}s\right] + C_{1}t,
\end{equation*}
\noindent which can be rewritten as
\begin{equation*}
\mathbb{E}\left[\int_{0}^{t}\left(\epsilon_{x, s}^{2} + \epsilon_{y, s}^{2} + \epsilon_{z, s}^{2}\right)\mathrm{d}s\right] \leq \frac{V\left(\xi_{0}\right)}{- C_{2}} + \frac{C_{1}}{- C_{2}}t.
\end{equation*}

\noindent Hence, taking limit and letting $t \to \infty$, we finally obtain the desired estimate
\begin{equation*}
\limsup_{t \to \infty}\frac{1}{t}\mathbb{E}\left[\int_{0}^{t}\left(\left(X_{s} - x^{\ast}\right)^{2} + \left(Y_{s} - y^{\ast}\right)^{2} + \left(Z_{s} - z^{\ast}\right)^{2}\right)\mathrm{d}s\right] \leq C,
\end{equation*}
\noindent thus concluding the proof.
\end{proof}

\begin{remark}\label{remark2}
Note that in the reasoning of the previous theorem, all terms can be controlled except for the parabolic one.

It may happen that the parabola has no real roots and remains always negative. From its expression, this occurs when the equilibrium lies sufficiently far from the fold curves $\Gamma^{-}$ or $\Gamma^{+}$ (in the $x^{\ast}$ component). In this case, the constant $C_{1}$ depends only on $K$. Moreover, in the deterministic setting, that is, when $K$ can be taken equal to zero, the obtained result shows that the equilibrium point is globally asymptotically stable.

In general, since the parabola depends on $x^{\ast}$, the distance between the equilibrium point and the fold curves can be interpreted as a measure of the stability of the point: as this distance increases, i.e., as $\left|x^{\ast}\right|$ grows, the constant $C_{1}$ decreases, and consequently so does $C$. This reflects the fact that equilibria located deeper inside the stable sheets $S_{l}$ or $S_{r}$ exhibit stronger stability properties.

On the other hand, a similar behaviour is observed with respect to the constant $K$: the smaller the value of $K$, the smaller the constant appearing in condition \eqref{condi}. Since $K$ controls the intensity of the stochastic perturbation, this is consistent with the fact that, as the noise level decreases, the stochastic trajectories remain closer, on average, to the deterministic equilibrium.

As will be shown in the numerical section, even in the case where the deterministic system possesses a stable equilibrium point (for instance when it lies on $S_{l}$ or $S_{r}$), the stochastic solution may jump to $S_{m}$. This event becomes less likely as $\left|x^{\ast}\right|$ increases and as the noise intensity $K$ decreases. In agreement with estimate \eqref{condi}, we will observe numerically that the constant appearing in the bound increases with the noise intensity.
\end{remark}

\section{Numerical Simulations}\label{sec4}
In this section, we illustrate several qualitative behaviours of the stochastic ICC model through numerical simulations that do not arise in the deterministic framework. In particular, stochastic perturbations may induce intermittent large excursions in the attractive regime, generate mixed-mode oscillations in parameter regions where the deterministic system exhibits smooth relaxation oscillations, and destroy the strict periodicity of deterministic MMOs.

All simulations are performed for the linear stochastic model \eqref{lineal}, which preserves the slow-fast structure of the deterministic ICC system.

Since the ICC model is a stiff slow-fast system, numerical schemes with good stability properties are required. Classical It{\^o}--Taylor methods, such as the Euler-Maruyama scheme, are not well suited to accurately capture this type of multiscale dynamics for practical time steps. For this reason, we employ the stochastic TR-BDF2 method introduced in \cite{Caraballo2026}, which attains strong order two and preserves the favourable stability properties of its deterministic counterpart.

For completeness, we consider a uniform partition of the interval $\left[0, T\right]$, with $t_{0} = 0 < t_{1} = t_{0} + h < \cdots < t_{n} = t_{n - 1} + h < t_{N} = T$, for some $N \in \mathbb{N}$ and time step $h = \frac{T}{N}$. We recall that the stochastic TR-BDF2 method applied to the general SDE \eqref{sde} is written as \cite{Caraballo2026}
\begin{equation*}
\left\{\begin{array}{l}
\displaystyle \text{Given } Y_{0}, \text{ compute } \forall n = 0, \dots, N - 1,\\
\displaystyle Y_{n + \gamma} = Y_{n} + \frac{1}{2}\left(f\left(t_{n}, Y_{n}\right) + f\left(t_{n + \gamma}, Y_{n + \gamma}\right)\right)h\gamma + \hat{P}_{n}\\
\displaystyle Y_{n + 1} = \gamma_{3}Y_{n + \gamma} + \left(1 - \gamma_{3}\right)Y_{n} + f\left(t_{n + 1}, Y_{n + 1}\right)h\gamma_{2} + \hat{Q}_{n} + \hat{Q}_{n + \gamma},
\end{array}\right.
\end{equation*}
\noindent where $Y_{n}$ is the approximation of $U_{t_{n}}$, $Y_{n + \gamma}$ is an auxiliary step, $t_{n + \gamma} = t_{n} + h\gamma$, $\gamma \in \left(0, 1\right)$, $\gamma_{2} = \frac{1 - \gamma}{2 - \gamma}$ and $\gamma_{3} = \frac{1}{\gamma\left(2 - \gamma\right)}$. Moreover, $\hat{P}_{n}$, $\hat{Q}_{n}$ and $\hat{Q}_{n + \gamma}$ are the stochastic correction terms involving stochastic integrals and terms arising from the It{\^o}--Taylor expansion of order two, required to achieve strong order two; see \cite{Caraballo2026} for details.

This scheme is a stochastic extension of the deterministic TR-BDF2 method (if $g \equiv 0$, then $\hat{P}_{n}$, $\hat{Q}_{n}$ and $\hat{Q}_{n + \gamma}$ vanish). From now on, we take $\gamma = 2 - \sqrt{2}$, which minimizes the local truncation error of the deterministic TR-BDF2 scheme. We recall that the deterministic TR-BDF2 scheme is second-order accurate and is $A$-stable and $L$-stable; see \cite{Bank1985, Hairer2010, Hosea1996}.

The stochastic TR-BDF2 attains strong order two of convergence and inherits the favourable stability properties of its deterministic counterpart. In particular, it is $A$-stable for the additive noise test equation and mean-square stable for sufficiently small time steps. Moreover, numerical evidence suggests that it exhibits improved mean-square stability properties compared with the It{\^o}--Taylor approximation of order two.

All tests were performed on a personal computer (MacBook Air, M1, 8GB, 2020) using MATLAB$\_$R2024b software.

For the stochastic model \eqref{lineal}, we fix $\left(X_{0}, Y_{0}, Z_{0}\right)^{\top} = \left(- 1, - 4, 1\right)^{\top}$ as the initial condition. To analyse the effect of the stochastic perturbations, we perform simulations activating a single noise coefficient at a time, that is, taking $\sigma_{i} \neq 0$ and $\sigma_{j} = 0$ for $j \neq i$, with $i, j \in \left\{1, \dots, 4\right\}$. Since the qualitative behaviour obtained in these four cases is similar, we only display the results corresponding to the representative case in which $\sigma_{1} \neq 0$ and $\sigma_{2} = \sigma_{3} = \sigma_{4} = 0$.

In all experiments, the realizations of the driving Wiener processes are fixed in order to ensure a consistent comparison between scenarios. More precisely, the same realization of $W_{t}^{1}$ is used across all simulations involving the first noise component, and analogously for $W_{t}^{2}$. The two Wiener processes are generated independently.

\subsection{Noise-induced jumps in the attractive regime}
We first consider the attractive regime in which the deterministic ICC model admits a unique globally asymptotically stable equilibrium. For $\mu = 2.5$, all deterministic trajectories converge to this equilibrium (see Figure~\ref{Fig1a}).

In the stochastic framework, Theorem~\ref{th} provides bounds for the average oscillations around the deterministic equilibrium. Nevertheless, as noted in Remark~\ref{remark2}, stochastic perturbations may occasionally induce qualitative changes in the dynamics. In particular, random fluctuations may drive the trajectory from the stable sheets $S_{l}$ or $S_{r}$ of the critical manifold into the unstable middle sheet $S_{m}$. Owing to the slow-fast structure of the system, once such a transition occurs the solution undergoes a fast excursion before returning to a neighbourhood of the equilibrium, producing a noise-induced jump that is absent in the deterministic attractive regime.

This behaviour is illustrated in Figure~\ref{fig:22}, where we show simulations for the representative case with a single active noise component. For the noise intensity shown in the left panel, the solution stays close to equilibrium for long time intervals and only sporadically exhibits large excursions. When the noise level is doubled (right panel), these transitions become more frequent and the intermittent character of the dynamics becomes more pronounced.

\begin{figure}[ht]
\centering
\begin{subfigure}[ht]{0.45\textwidth}
\centering
\includegraphics[scale=0.28]{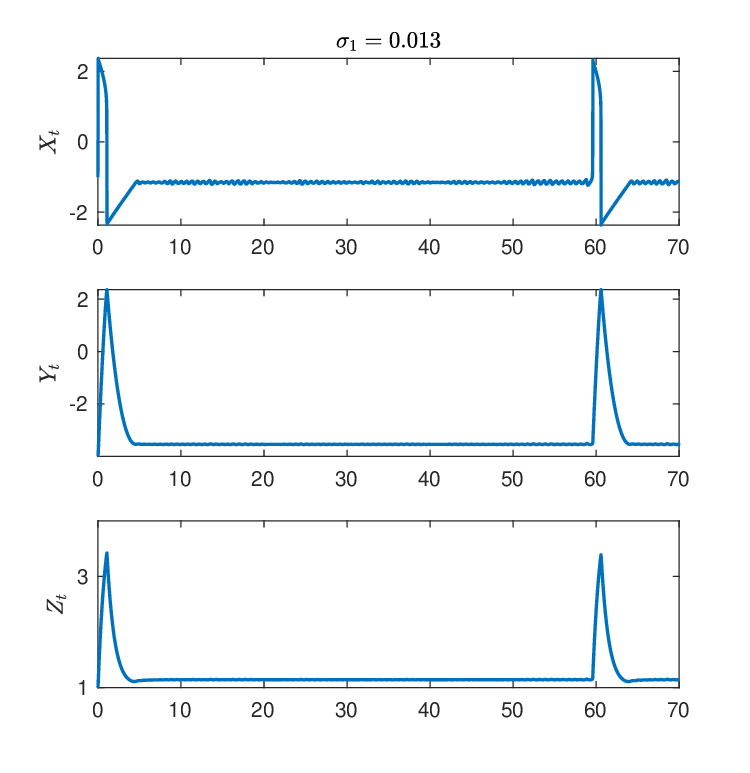}
\caption{$\sigma_{1} = 0.013$, $\sigma_{j} = 0$ for $j \neq 1$.}
\end{subfigure}
\hfill
\begin{subfigure}[ht]{0.45\textwidth}
\centering
\includegraphics[scale=0.28]{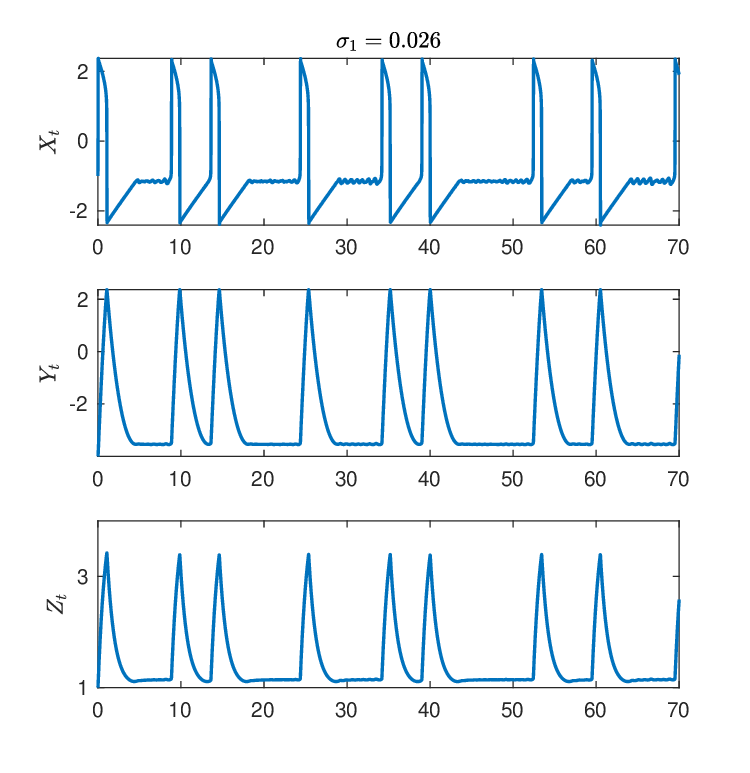}
\caption{$\sigma_{1} = 0.026$, $\sigma_{j} = 0$ for $j \neq 1$.}
\end{subfigure}
\caption{Noise-induced jumps in the attractive regime ($\mu = 2.5$) for model \eqref{lineal}. Left: reference noise intensity. Right: doubled noise intensity.}\label{fig:22}
\end{figure}

The value of $\sigma_{1}$ in the left panel corresponds to the smallest intensity for which a jump is observed within the time interval $\left[0, T\right]$ (with $T = 70$) under the fixed realizations of the Wiener processes described above. Depending on the particular realization and on the observation window, such excursions may occur earlier, later, or not at all.

These simulations show that, although the deterministic dynamics enforces convergence towards a stable equilibrium, stochastic perturbations allow the system to explore dynamically inaccessible regions of phase space. Moreover, the frequency of large excursions increases with the noise intensity, in agreement with the bounds in Theorem~\ref{th}.

\subsection{Noise effects in the unstable equilibrium regime}\label{nume}
We next consider parameter regimes in which the deterministic ICC model exhibits an unstable equilibrium. Two representative cases are shown in Figure~\ref{fig:periodic2}. For $\mu = 2.4$, the equilibrium lies on the middle sheet $S_{m}$ close to a fold curve and the deterministic dynamics generates strictly periodic mixed-mode oscillations (MMOs), as shown in Figure~\ref{Fig1c2}. For $\mu = 2.21$, the deterministic system instead converges to a smooth relaxation limit cycle, as illustrated in Figure~\ref{Fig1b}.

In the stochastic framework, these behaviours are qualitatively modified. For $\mu = 2.4$ (left panel of Figure~\ref{fig:periodic2}), stochastic perturbations destroy the strict periodicity of deterministic MMOs. Although the alternating structure of small- and large-amplitude oscillations is preserved, the duration of the small-amplitude phase and the timing of the large excursions vary from cycle to cycle, and the number of small oscillations preceding each large excursion is no longer constant.

\begin{figure}[ht]
\centering
\begin{subfigure}[ht]{0.45\textwidth}
\centering
\includegraphics[scale=0.28]{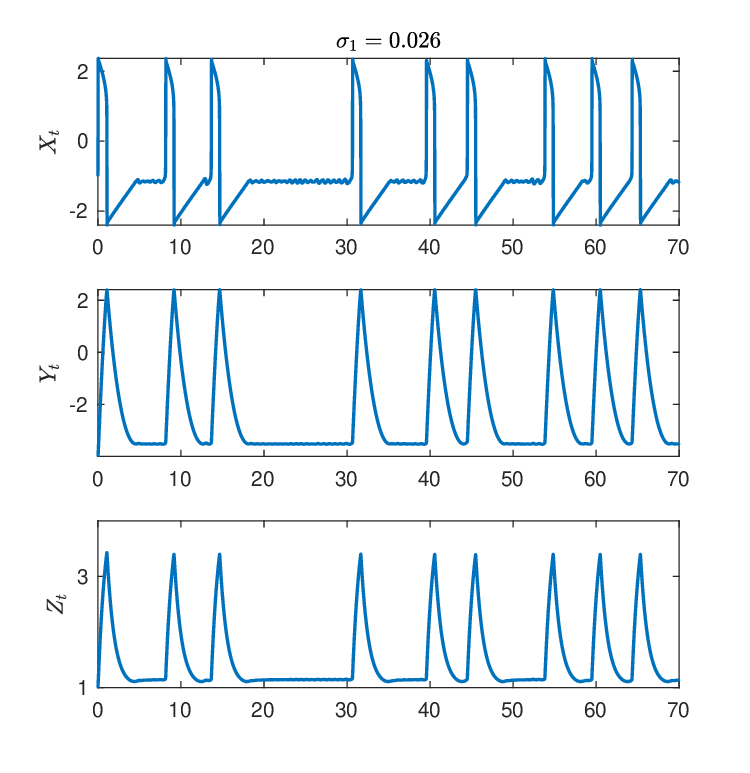}
\caption{$\sigma_{1} = 0.026$, $\sigma_{j} = 0$ for $j \neq 1$.}
\end{subfigure}
\hfill
\begin{subfigure}[ht]{0.45\textwidth}
\centering
\includegraphics[scale=0.28]{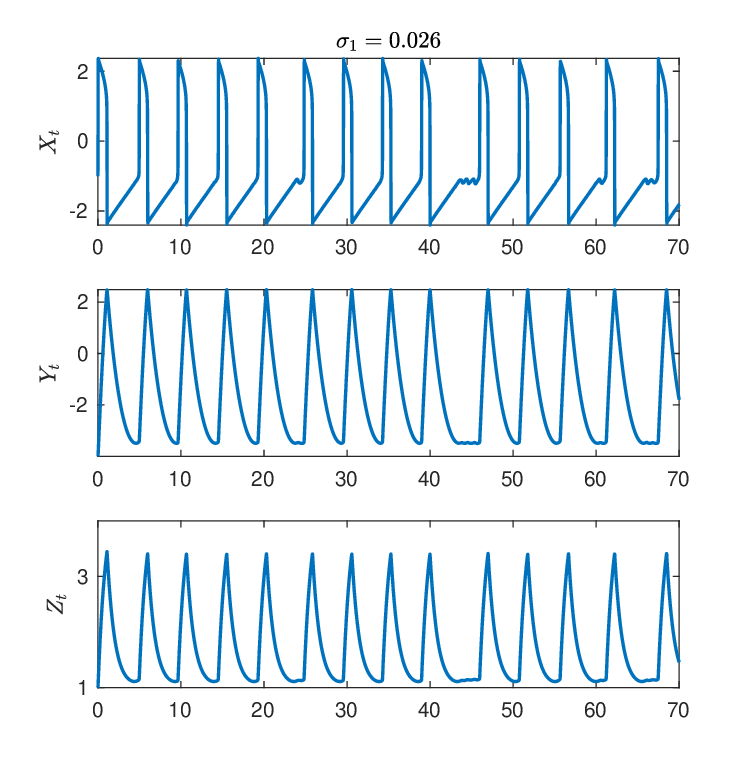}
\caption{$\sigma_{1} = 0.026$, $\sigma_{j} = 0$ for $j \neq 1$.}
\end{subfigure}
\caption{Noise effects in the unstable equilibrium regime for model \eqref{lineal}. Left: loss of periodicity of MMOs ($\mu = 2.4$). Right: noise-induced MMOs in the relaxation regime ($\mu = 2.21$).}\label{fig:periodic2}
\end{figure}

For $\mu = 2.21$ (right panel), stochastic forcing induces mixed-mode oscillations in a parameter regime where the deterministic model exhibits only smooth relaxation oscillations. The resulting trajectories display irregular sequences of small oscillations followed by large excursions, reflecting the influence of the noise.

As in the previous subsection, we show simulations for the representative case with a single active noise component, which clearly illustrates the qualitative changes in the oscillatory patterns.

These simulations show that stochastic perturbations may both destroy deterministic periodicity and induce mixed-mode dynamics in parameter regimes where such behaviour is absent in the deterministic model.

\subsection{Numerical validation of the theoretical estimates}
We now provide a numerical validation of the estimate obtained in Theorem~\ref{th} and discussed in Remark~\ref{remark2}. In the deterministic attractive regime, the limit in \eqref{condi} tends to zero when the equilibrium is globally asymptotically stable, meaning that the time-averaged squared distance to the equilibrium vanishes.

In the stochastic framework, this limit is no longer expected to vanish; however, for sufficiently small noise intensities the solution should remain close, on average, to its deterministic counterpart over long time intervals.

To approximate numerically the quantity in \eqref{condi}, we consider a uniform partition of the interval $\left[0, T\right]$, given by $t_{0} = 0 < t_{1} = t_{0} + h < \cdots < t_{n} = t_{n - 1} + h < t_{N} = T$, for some $N \in \mathbb{N}$ and time step $h = \frac{T}{N}$. For each realization of the Wiener processes, we compute the discrete-time average of the squared distance to the deterministic equilibrium and then average over $M$ independent realizations. More precisely, we define the following empirical estimator:
\begin{equation}
\begin{array}{ll}
\displaystyle \hat{C}_{t_{n}} \coloneqq \frac{1}{t_{n}}\left(\frac{1}{M}\sum_{j = 1}^{M}\left(h\sum_{k = 1}^{n}\left\|\hat{\xi}_{t_{k}}^{j} - \left(x^{\ast}, y^{\ast}, z^{\ast}\right)^{\top}\right\|^{2}\right)\right), & \displaystyle n = 1, \dots, N,
\end{array}
\label{Ctn}
\end{equation}
\noindent which provides an empirical approximation of the bound appearing in \eqref{condi}. Here, $\hat{\xi}_{t_{k}}^{j}$ denotes the numerical approximation of the solution at time $t_{k}$ corresponding to the $j$-th realization of the Wiener processes.

We first consider the deterministic case in the attractive regime $\mu = 2.5$, using a large final time $T = 1000$ in order to approximate the asymptotic behaviour predicted by the theorem. The evolution of $\hat{C}_{t_{n}}$ is displayed in Figure~\ref{fig:detC}. The obtained value at the final time is $\hat{C}_{T} \approx 0.051771$, which is small and decreases as $T$ increases. This is consistent with the theoretical prediction that the limit in \eqref{condi} vanishes in the deterministic setting.

We next perform the same computation in the stochastic framework, considering $M = 50$ independent realizations and increasing values of the noise intensity $\sigma_{1}$ while setting $\sigma_{2} = \sigma_{3} = \sigma_{4} = 0$. The results are shown in Figure~\ref{fig:stochC}.

\begin{figure}[ht]
\centering
\begin{subfigure}[ht]{0.45\textwidth}
\centering
\includegraphics[scale=0.32]{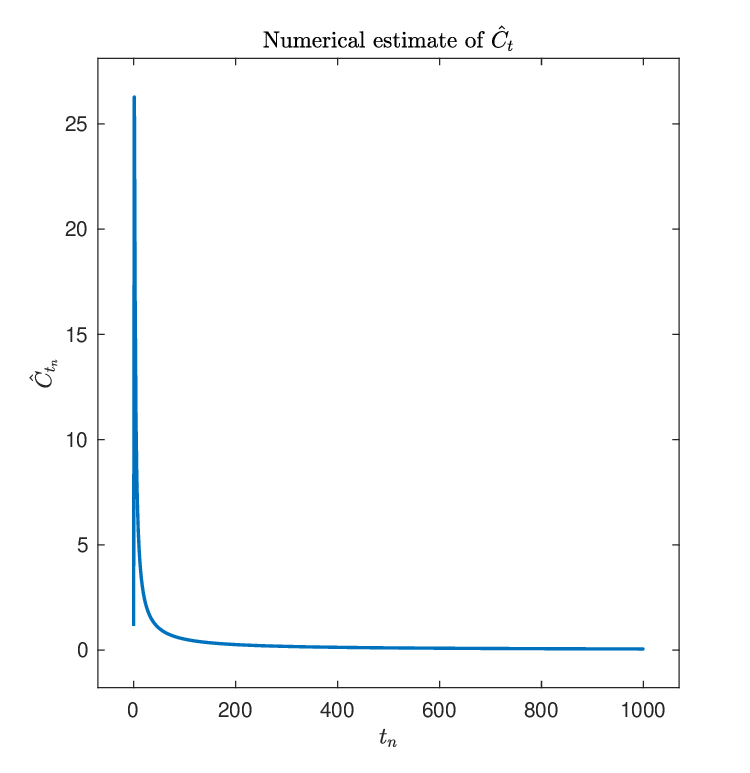}
\caption{Deterministic case ($\mu = 2.5$).}\label{fig:detC}
\end{subfigure}
\hfill
\begin{subfigure}[ht]{0.45\textwidth}
\centering
\includegraphics[scale=0.32]{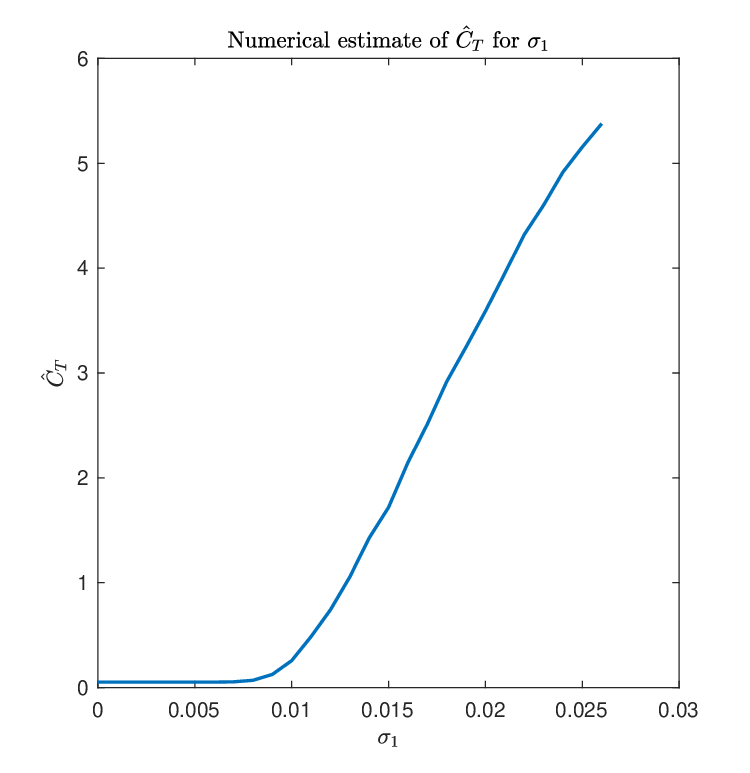}
\caption{Stochastic case with $\sigma_{1} \in \left[0, 0.026\right]$.}\label{fig:stochC}
\end{subfigure}
\caption{Numerical approximation of the quantity in \eqref{condi} in the attractive regime ($\mu = 2.5$). Left: deterministic model. Right: stochastic model with increasing noise intensity $\sigma_{1}$.}
\end{figure}

As expected, the estimated quantity increases with the noise intensity, in agreement with the bound in \eqref{condi}, where the constant depends explicitly on the noise level. 

Moreover, for larger values of $\mu$, the modulus of the equilibrium component $\left|x^{\ast}\right|$ increases, making noise-induced excursions less likely and leading to smaller empirical values. This is consistent with the discussion in Remark~\ref{remark2}, where the position of the equilibrium relative to the fold curves is shown to influence the magnitude of the bound in \eqref{condi}.

\section{The 6D ICC model for the dynamics in two coupled cells}\label{sec5}
To study a more realistic setting, we extend the 3D model \eqref{ICC} to a pair of coupled cells interacting through a coupling term acting on the recovery variables and driven by the fast voltages. This framework can in fact be extended to two connected clusters of cells, exhibiting a behaviour similar to that of two coupled neurons \cite{Bandera2022}. We therefore first study the stochastic case of two coupled cells, starting from the corresponding deterministic model:
\begin{equation}
\left\{\begin{array}{l}
\displaystyle \left.\begin{array}{l}
\displaystyle \dot{x}_{1} = \tau\left(b_{0}y_{1} + b_{1}x_{1} - x_{1}^{3} - \phi_{f}\left(z_{1}\right)\right),\\
\displaystyle \dot{y}_{1} = \tau\varepsilon k_{1}\left(a_{0}x_{1} + a_{1}y_{1} + a_{2} + c\left(x_{1} - x_{2}\right)\right),\\
\displaystyle \dot{z}_{1} = \tau\varepsilon\left(\phi_{r}\left(x_{1}\right) - \frac{z_{1} - z_{b}}{\tau_{z}}\right),
\end{array}\displaystyle \right\} O_{1}\\
\displaystyle \left.\begin{array}{l}
\displaystyle \dot{x}_{2} = \tau\left(b_{0}y_{2} + b_{1}x_{2} - x_{2}^{3} - \phi_{f}\left(z_{2}\right)\right),\\
\displaystyle \dot{y}_{2} = \tau\varepsilon k_{2}\left(a_{0}x_{2} + a_{1}y_{2} + a_{2} + c\left(x_{2} - x_{1}\right)\right),\\
\displaystyle \dot{z}_{2} = \tau\varepsilon\left(\phi_{r}\left(x_{2}\right) - \frac{z_{2} - z_{b}}{\tau_{z}}\right),
\end{array}\displaystyle \right\} O_{2}\\
\displaystyle \left(x_{1}\left(0\right), y_{1}\left(0\right), z_{1}\left(0\right), x_{2}\left(0\right), y_{2}\left(0\right), z_{2}\left(0\right)\right)^{\top} \text{ is given}.
\end{array}\right.
\label{6d}
\end{equation}

Here $O_{1}$ and $O_{2}$ denote the models for the cells with variables $\left(x_{1}, y_{1}, z_{1}\right)^{\top}$ and $\left(x_{2}, y_{2}, z_{2}\right)^{\top}$ respectively. The parameters $k_{i} > 0$, $i = 1, 2$, account for possible heterogeneity between cells, while $c \in \left[- 1, 1\right]$ measures the coupling strength. Note that when $c = 0$ the interaction vanishes and each cell behaves independently as described in Section~\ref{sec2}.

From now on, we are mainly interested in the homogeneous case $k_{1} = k_{2} = 1$, which is studied in \cite{Fern2020}. The heterogeneous case $k_{i} \neq k_{j}$ and generalizations to larger networks are considered in \cite{Bandera2022}. In addition, in \cite{Bandera2026} different values of the coupling parameter across cells have been explored, leading to the appearance of further dynamical behaviours. This setting is also of considerable interest and will be addressed in future work.

However, the theoretical results are established for general $k_{i}$, $i = 1, 2$, and can likewise be extended to the case of distinct $c$. Our main objective is to analyse how stochastic perturbations affect the qualitative behaviours of the deterministic model. We fix the parameters of Section~\ref{sec2} for $\mu = 2.4$, where the uncoupled system exhibits MMOs.

In this context, we summarize the main asymptotic behaviours, which now depend on the coupling parameter $c$, considering the initial conditions $\left(r, 4r - r^{3}, 1\right)^{\top}$ with $r = - 1.25$ for $O_{1}$ and $r = - 1.75$ for $O_{2}$ (see \cite{Fern2020}):

\begin{enumerate}
\item \textit{Total oscillation death.} Both cells converge to a stable equilibrium for $c < - 0.509$ (Figure~\ref{Fig2a}, $c = - 0.514$).

\item \textit{Relaxation loss.} One cell produces trajectories that approach both outer sheets $S_{l}$ and $S_{r}$, while the other remains confined near $S_{l}$ performing small oscillations, for $c \in \left[- 0.509, - 0.501\right]$ (Figure~\ref{Fig2b}, $c = - 0.501$).

\item \textit{Antiphase synchronization.} Both cells converge asymptotically to the same periodic orbit with an approximate half-period phase shift between them. This occurs for $c \in \left(- 0.501, 0\right)$ (Figure~\ref{Fig2c}, $c = - 0.005$).

\item \textit{Uncoupled case.} For $c = 0$, both evolve independently, converging to the MMO regime of Section~\ref{sec2} (Figure~\ref{Fig2d}).

\item \textit{Almost-in-phase synchronization.} Both cells synchronize with a small phase shift, following an attractive relaxation limit cycle that generates MMOs. This regime is observed for $c \in \left(0, 0.716\right]$ (Figure~\ref{Fig2e}, $c = 0.1$).

\item \textit{In-phase locking synchronization.} Both cells converge asymptotically to the same periodic orbit with zero phase difference, coinciding with the cycle of the uncoupled case. This occurs for $c > 0.716$ (Figure~\ref{Fig2f}, $c = 1$).
\end{enumerate}

These regimes show that the coupling may suppress oscillations, reorganize the relaxation structure, or produce different synchronization patterns. In the following subsection we introduce stochastic perturbations of \eqref{6d} and analyse how noise modifies these deterministic behaviours.

\begin{figure}[ht]
\centering
\begin{subfigure}[ht]{0.45\textwidth}
\centering
\includegraphics[scale=0.28]{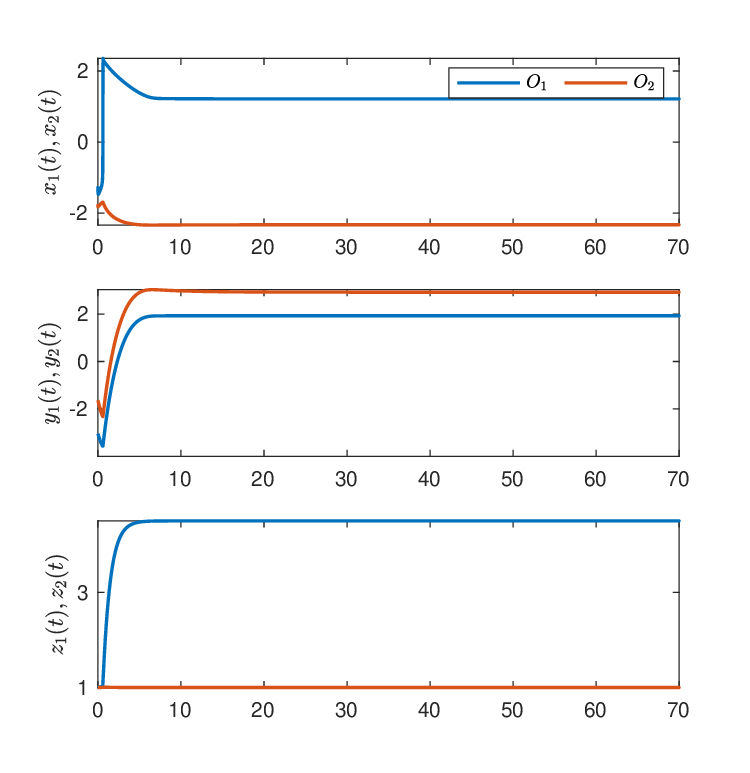}
\caption{Total oscillation death, $c = - 0.514$.}\label{Fig2a}
\end{subfigure}
\hfill
\begin{subfigure}[ht]{0.45\textwidth}
\centering
\includegraphics[scale=0.28]{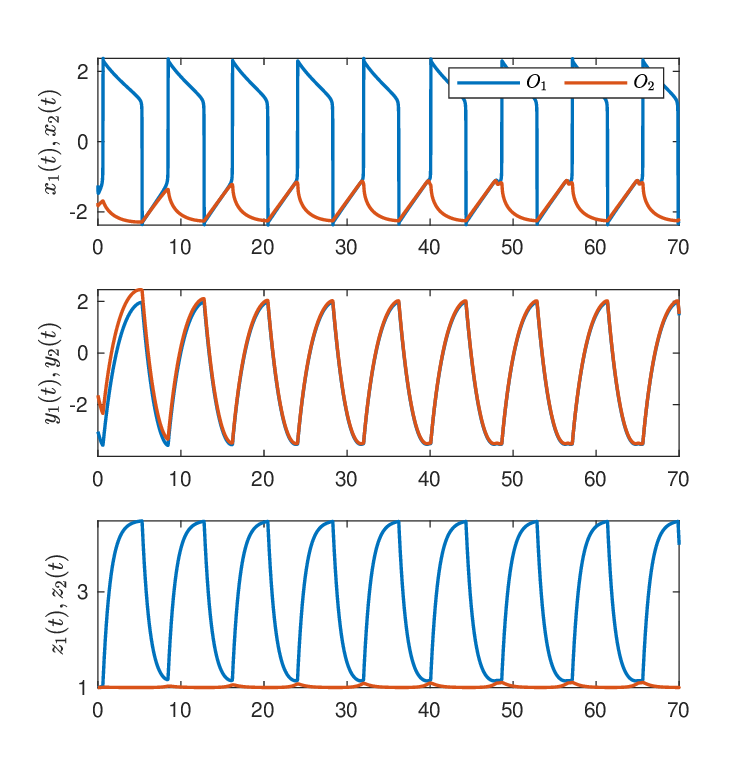}
\caption{Relaxation loss, $c = - 0.501$.}\label{Fig2b}
\end{subfigure}
\begin{subfigure}[ht]{0.45\textwidth}
\centering
\includegraphics[scale=0.28]{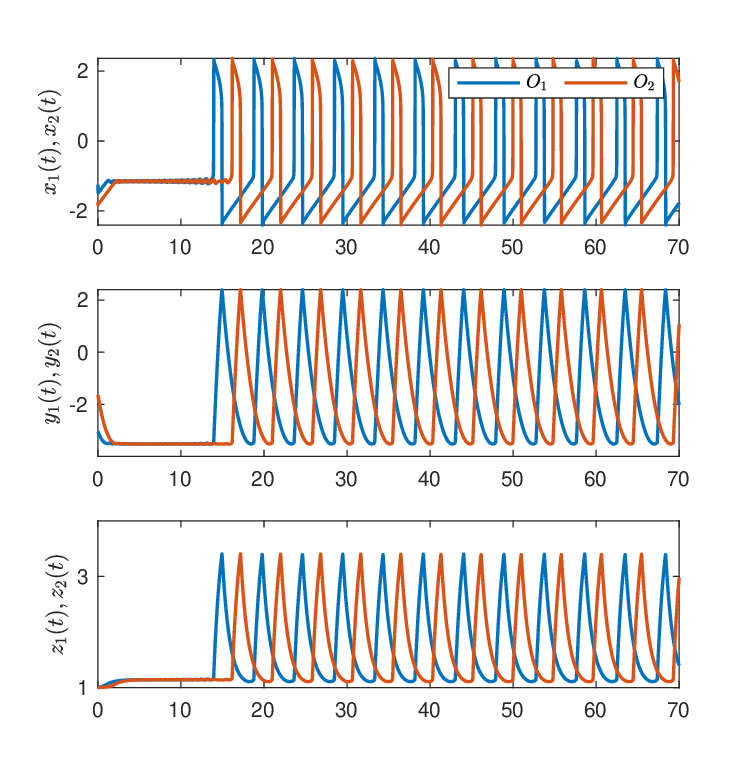}
\caption{Antiphase synchronization, $c = - 0.005$.}\label{Fig2c}
\end{subfigure}
\hfill
\begin{subfigure}[ht]{0.45\textwidth}
\centering
\includegraphics[scale=0.28]{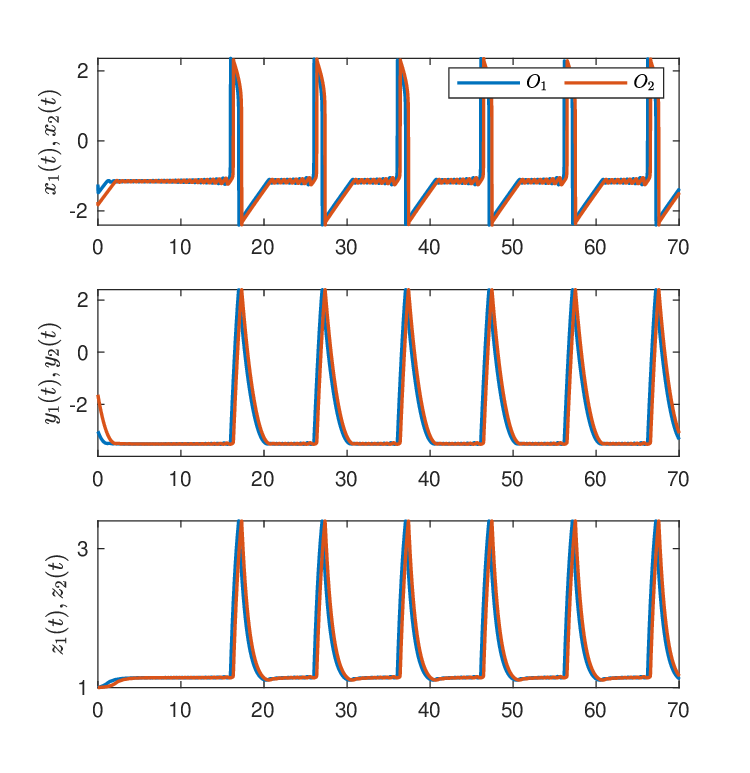}
\caption{Uncoupled case, $c = 0$.}\label{Fig2d}
\end{subfigure}
\caption{Asymptotic behaviours of the solution of model \eqref{6d} for $c \leq 0$.}
\end{figure}

\begin{figure}[ht]
\centering
\begin{subfigure}[ht]{0.45\textwidth}
\centering
\includegraphics[scale=0.28]{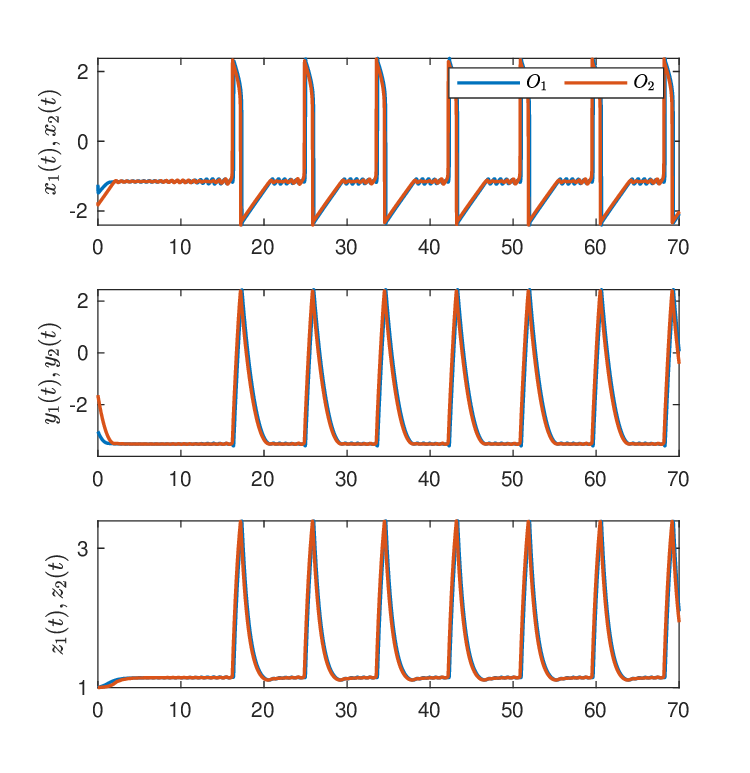}
\caption{Almost-in-phase synchronization, $c = 0.1$.}\label{Fig2e}
\end{subfigure}
\hfill
\begin{subfigure}[ht]{0.45\textwidth}
\centering
\includegraphics[scale=0.28]{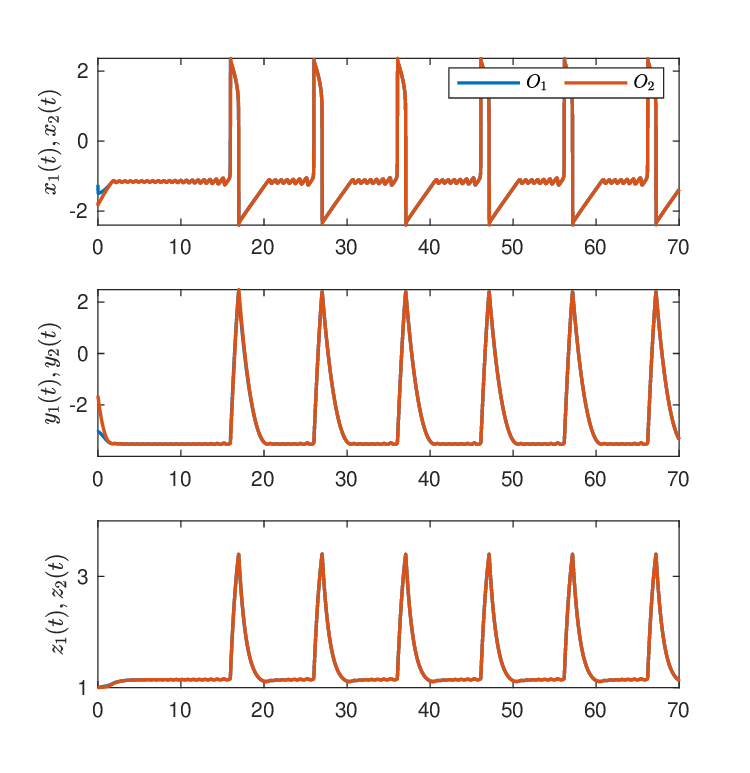}
\caption{In-phase locking synchronization, $c = 1$.}\label{Fig2f}
\end{subfigure}
\caption{Asymptotic behaviours of the solution of model \eqref{6d} for $c > 0$.}
\end{figure}

\subsection{Stochastic formulation of the 6D model}
To make the coupled two-cell ICC model more realistic, we now incorporate stochastic perturbations into the six-dimensional system. As in the single-cell case, this allows us to account for random fluctuations arising from intrinsic variability and uncertain physiological effects that are not captured by the purely deterministic formulation. Moreover, the stochastic framework makes it possible to observe qualitative behaviours that do not arise in the deterministic setting, such as noise-induced transitions between dynamical regimes or modifications in the synchronization patterns of the coupled cells.

Let $\Xi_{t} = \left(X_{1, t}, Y_{1, t}, Z_{1, t}, X_{2, t}, Y_{2, t}, Z_{2, t}\right)^{\top}$. We consider the following stochastic version of the coupled model \eqref{6d}:
\begin{equation}
\left\{\begin{array}{l}
\displaystyle \mathrm{d}X_{1, t} = \tau\left(b_{0}Y_{1, t} + b_{1}X_{1, t} - X_{1, t}^{3} - \phi_{f}\left(Z_{1, t}\right)\right)\mathrm{d}t + G_{1}\left(\Xi_{t}\right)\mathrm{d}W_{t}^{1},\\
\displaystyle \mathrm{d}Y_{1, t} = \tau\varepsilon k_{1}\left(a_{0}X_{1, t} + a_{1}Y_{1, t} + a_{2} + c\left(X_{1, t} - X_{2, t}\right)\right)\mathrm{d}t + \sqrt{\varepsilon}G_{2}\left(\Xi_{t}\right)\mathrm{d}W_{t}^{2},\\
\displaystyle \mathrm{d}Z_{1, t} = \tau\varepsilon\left(\phi_{r}\left(X_{1, t}\right) - \frac{Z_{1, t} - z_{b}}{\tau_{z}}\right)\mathrm{d}t,\\
\displaystyle \mathrm{d}X_{2, t} = \tau\left(b_{0}Y_{2, t} + b_{1}X_{2, t} - X_{2, t}^{3} - \phi_{f}\left(Z_{2, t}\right)\right)\mathrm{d}t + G_{3}\left(\Xi_{t}\right)\mathrm{d}W_{t}^{3},\\
\displaystyle \mathrm{d}Y_{2, t} = \tau\varepsilon k_{2}\left(a_{0}X_{2, t} + a_{1}Y_{2, t} + a_{2} + c\left(X_{2, t} - X_{1, t}\right)\right)\mathrm{d}t + \sqrt{\varepsilon}G_{4}\left(\Xi_{t}\right)\mathrm{d}W_{t}^{4},\\
\displaystyle \mathrm{d}Z_{2, t} = \tau\varepsilon\left(\phi_{r}\left(X_{2, t}\right) - \frac{Z_{2, t} - z_{b}}{\tau_{z}}\right)\mathrm{d}t,\\
\displaystyle \Xi_{0} \text{ is given},
\end{array}\right.
\label{S6d}
\end{equation}
\noindent where $W_{t}^{1}$, $W_{t}^{2}$, $W_{t}^{3}$ and $W_{t}^{4}$ are independent scalar Wiener processes. The functions $G_{1}$, $G_{2}$, $G_{3}$ and $G_{4}$ describe the intensity of the stochastic perturbations and, as in the three-dimensional case, they are assumed to satisfy suitable regularity and growth conditions ensuring existence and uniqueness of solutions, as well as the asymptotic estimates established below.

In particular, for the numerical simulations we will focus on the case in which the stochastic perturbations act on the linear terms of the first two equations of each cell. Moreover, since we take the same linear coefficients in both cells, the same Wiener processes naturally appear in both equations. More precisely, by formally replacing the linear parameters with stochastic perturbations as in the three-dimensional case and replacing $\varepsilon c$ with $\varepsilon c + \sqrt{\varepsilon}\frac{\sigma_{5}}{\tau k}\frac{\mathrm{d}W_{t}^{2}}{\mathrm{d}t}$, we obtain the following six-dimensional linear stochastic model:
\begin{equation}
\left\{\begin{array}{l}
\displaystyle \mathrm{d}X_{1, t} = \tau\left(b_{0}Y_{1, t} + b_{1}X_{1, t} - X_{1, t}^{3} - \phi_{f}\left(Z_{1, t}\right)\right)\mathrm{d}t + \left(\sigma_{1}X_{1, t} + \sigma_{2}Y_{1, t}\right)\mathrm{d}W_{t}^{1},\\
\displaystyle \mathrm{d}Y_{1, t} = \tau\varepsilon k_{1}\left(a_{0}X_{1, t} + a_{1}Y_{1, t} + a_{2} + c\left(X_{1, t} - X_{2, t}\right)\right)\mathrm{d}t\\
\displaystyle \mathrel{\phantom{\mathrm{d}Y_{1, t} = }} + \sqrt{\varepsilon}\left(\sigma_{3}X_{1, t} + \sigma_{4}Y_{1, t} + \sigma_{5}\left(X_{1, t} - X_{2, t}\right)\right)\mathrm{d}W_{t}^{2},\\
\displaystyle \mathrm{d}Z_{1, t} = \tau\varepsilon\left(\phi_{r}\left(X_{1, t}\right) - \frac{Z_{1, t} - z_{b}}{\tau_{z}}\right)\mathrm{d}t,\\
\displaystyle \mathrm{d}X_{2, t} = \tau\left(b_{0}Y_{2, t} + b_{1}X_{2, t} - X_{2, t}^{3} - \phi_{f}\left(Z_{2, t}\right)\right)\mathrm{d}t + \left(\sigma_{1}X_{2, t} + \sigma_{2}Y_{2, t}\right)\mathrm{d}W_{t}^{1},\\
\displaystyle \mathrm{d}Y_{2, t} = \tau\varepsilon k_{2}\left(a_{0}X_{2, t} + a_{1}Y_{2, t} + a_{2} + c\left(X_{2, t} - X_{1, t}\right)\right)\mathrm{d}t\\
\displaystyle \mathrel{\phantom{\mathrm{d}Y_{2, t} = }} + \sqrt{\varepsilon}\left(\sigma_{3}X_{2, t} + \sigma_{4}Y_{2, t} + \sigma_{5}\left(X_{2, t} - X_{1, t}\right)\right)\mathrm{d}W_{t}^{2},\\
\displaystyle \mathrm{d}Z_{2, t} = \tau\varepsilon\left(\phi_{r}\left(X_{2, t}\right) - \frac{Z_{2, t} - z_{b}}{\tau_{z}}\right)\mathrm{d}t,\\
\displaystyle \Xi_{0} \text{ is given},
\end{array}\right.
\label{lineal6d}
\end{equation}

This particular choice preserves the slow-fast structure of the deterministic coupled system and provides a natural extension of the linear stochastic perturbations considered in the single-cell framework.

We now state the main theoretical results for the stochastic coupled model \eqref{S6d}. These results extend those obtained for the three-dimensional case. Their proofs follow essentially the same arguments, the only difference being the presence of the additional coupling terms in the equations for $Y_{1, t}$ and $Y_{2, t}$.

As a first result, we show that the stochastic coupled model preserves the positivity of the calcium variables whenever the initial conditions are positive.

\begin{theorem}[Positivity of the variables $Z_{1, t}$ and $Z_{2, t}$]
Let $Z_{1, 0}, Z_{2, 0} \in \mathbb{R}_{+}$. Then, the third and sixth components associated with model \eqref{S6d} satisfy
\begin{equation*}
\begin{array}{ll}
\displaystyle Z_{1, t}, Z_{2, t} \in \mathbb{R}_{+}, & \displaystyle \forall t \in \left[0, T\right].
\end{array}
\end{equation*}
\end{theorem}

\begin{proof}
The result follows exactly as in the three-dimensional case, since the equations for $Z_{1, t}$ and $Z_{2, t}$ are deterministic once $X_{1, t}$ and $X_{2, t}$ are given, and can be solved explicitly. Arguing as in the proof of Theorem~\ref{posi} for the variable $Z_{t}$ in the single-cell model, one obtains that both $Z_{1, t}$ and $Z_{2, t}$ remain positive for all $t \in \left[0, T\right]$ whenever $Z_{1, 0}, Z_{2, 0} > 0$.
\end{proof}

We next state an existence and uniqueness result for the stochastic coupled system.

\begin{theorem}[Existence and uniqueness for the 6D stochastic model]\label{EU6D}
Assume that the functions $G_{1}$, $G_{2}$, $G_{3}$ and $G_{4}$ are locally Lipschitz and satisfy the global linear growth condition
\begin{equation*}
\frac{1}{2}G_{1}\left(\Xi\right)^{2} + \frac{1}{2}\varepsilon G_{2}\left(\Xi\right)^{2} + \frac{1}{2}G_{3}\left(\Xi\right)^{2} + \frac{1}{2}\varepsilon G_{4}\left(\Xi\right)^{2} \leq K\left(1 + \left\|\Xi\right\|^{2}\right),
\end{equation*}
\noindent for all $\Xi \in \mathcal{R} \times \mathcal{R}$, for some constant $K > 0$.

Then, for any given initial value $\Xi_{0} \in \mathcal{R} \times \mathcal{R}$, the stochastic coupled model \eqref{S6d} admits a unique strong solution
\begin{equation*}
\begin{array}{ll}
\displaystyle \Xi_{t} \in \mathcal{M}^{2}\left(\left[0, T\right]; \mathcal{R} \times \mathcal{R}\right), \displaystyle \forall t \in \left[0, T\right].
\end{array}
\end{equation*}
\end{theorem}

\begin{proof}
The proof follows by direct application of Theorem~\ref{eu}. The local Lipschitz property holds since the new coupled terms are linear, and the diffusion coefficients are locally Lipschitz by hypothesis.

It remains to verify the growth condition. Arguing as in the proof of Theorem~\ref{EUICC}, the additional coupling terms appearing in the equations for $Y_{1, t}$ and $Y_{2, t}$ give rise to the products
\begin{equation*}
\tau\varepsilon c\left(k_{1}X_{1, t}Y_{1, t} - k_{1}X_{2, t}Y_{1, t} + k_{2}X_{2, t}Y_{2, t} - k_{2}X_{1, t}Y_{2, t}\right),
\end{equation*}
\noindent which can be controlled by means of Young's inequality \eqref{ine}. Therefore, for some constant $K^{\ast} > 0$, these extra terms are bounded by $K^{\ast}\left(1 + \left\|\Xi_{t}\right\|^{2}\right)$, and the monotone condition follows together with the assumed growth bound on the diffusion terms. Hence, Theorem~\ref{eu} yields the existence and uniqueness of a strong solution to \eqref{S6d}.
\end{proof}

\begin{remark}
Existence and uniqueness also hold for the linear stochastic system \eqref{lineal6d}. Indeed, the assumptions of Theorem~\ref{EU6D} are fulfilled since its drift and diffusion coefficients are locally Lipschitz, and the diffusion terms satisfy a global linear growth condition
\begin{equation*}
K = \max\left\{\sigma_{1}^{2} + \frac{3}{2}\varepsilon\sigma_{3}^{2} + 6\varepsilon\sigma_{5}^{2}, \sigma_{2}^{2} + \frac{3}{2}\varepsilon\sigma_{4}^{2}\right\},
\end{equation*}
\end{remark}

Finally, we turn to the asymptotic behaviour of the stochastic coupled system. Since the deterministic coupled problem \eqref{6d} admits at least one equilibrium point for every value of the coupling parameter $c$, it is natural to formulate an asymptotic result describing the behaviour of the stochastic trajectories relative to a deterministic equilibrium.

\begin{theorem}
Assume that the problem \eqref{6d} has a unique globally asymptotically stable equilibrium $\Xi^{\ast}$ and functions $G_{1}$, $G_{2}$, $G_{3}$ and $G_{4}$ of the stochastic model \eqref{S6d} satisfy the following global linear growth bounds:
\begin{equation*}
\frac{1}{2}G_{1}\left(\Xi\right)^{2} + \frac{1}{2}\varepsilon G_{2}\left(\Xi\right)^{2} + \frac{1}{2}G_{3}\left(\Xi\right)^{2} + \frac{1}{2}\varepsilon G_{4}\left(\Xi\right)^{2} \leq K\left(1 + \left\|\Xi\right\|^{2}\right),
\end{equation*}
\noindent for all $\Xi \in \mathcal{R} \times \mathcal{R}$, with
\begin{equation*}
K < \frac{\tau\varepsilon}{2}\min\left\{- k_{1}a_{1}, - k_{2}a_{1}, \frac{1}{\tau_{z}}\right\}.
\end{equation*}

Then, for any given initial condition $\Xi_{0} \in \mathcal{R} \times \mathcal{R}$, there exists a positive constant $C$ such that the solution $\Xi_{t}$ of model \eqref{S6d} satisfies
\begin{equation}
\limsup_{t \to \infty}\frac{1}{t}\mathbb{E}\left[\int_{0}^{t}\left\|\Xi_{s} - \Xi^{\ast}\right\|^{2}\mathrm{d}s\right] \leq C.
\label{condi6d}
\end{equation}
\end{theorem}

\begin{proof}
\noindent Let us define
\begin{equation*}
\hat{V}\left(\Xi_{t}\right) = \frac{1}{2}\left\|\Xi_{t} - \Xi^{\ast}\right\|^{2}.
\end{equation*}

\noindent In contrast to Theorem~\ref{th}, we do not introduce a weighted coefficient $\alpha$ in front of the recovery variables, since it may happen that $a_{0} + c = 0$, and we prefer not to distinguish between different cases.

\noindent Denoting
\begin{equation*}
\Xi^{\ast} = \left(x_{1}^{\ast}, y_{1}^{\ast}, z_{1}^{\ast}, x_{2}^{\ast}, y_{2}^{\ast}, z_{2}^{\ast}\right)^{\top},
\end{equation*}
\noindent we write
\begin{equation*}
\begin{array}{lll}
\displaystyle \epsilon_{x_{1}, t} = X_{1, t} - x_{1}^{\ast}, & \displaystyle \epsilon_{y_{1}, t} = Y_{1, t} - y_{1}^{\ast}, & \displaystyle \epsilon_{z_{1}, t} = Z_{1, t} - z_{1}^{\ast},\\
\displaystyle \epsilon_{x_{2}, t} = X_{2, t} - x_{2}^{\ast}, & \displaystyle \epsilon_{y_{2}, t} = Y_{2, t} - y_{2}^{\ast}, & \displaystyle \epsilon_{z_{2}, t} = Z_{2, t} - z_{2}^{\ast}.
\end{array}
\end{equation*}

\noindent Repeating the same procedure as in the proof of Theorem~\ref{th}, and applying It{\^o}'s formula \eqref{Ito} to $\hat{V}\left(\Xi_{t}\right)$, we obtain an expression of the form
\begin{equation}
\mathrm{d}\hat{V} = L\hat{V}\mathrm{d}t + \epsilon_{x_{1}, t}G_{1}\left(\Xi_{t}\right)\mathrm{d}W_{t}^{1} + \epsilon_{y_{1}, t}\sqrt{\varepsilon}G_{2}\left(\Xi_{t}\right)\mathrm{d}W_{t}^{2} + \epsilon_{x_{2}, t}G_{3}\left(\Xi_{t}\right)\mathrm{d}W_{t}^{3} + \epsilon_{y_{2}, t}\sqrt{\varepsilon}G_{4}\left(\Xi_{t}\right)\mathrm{d}W_{t}^{4},
\label{dhatV}
\end{equation}
\noindent with
\begin{equation*}
\begin{array}{l}
\displaystyle L\hat{V}\left(\Xi_{t}\right) = \sum_{i = 1}^{2}\left(\hat{A}_{i}\left(X_{i, t}\right)\epsilon_{x_{i}, t}^{2} + \hat{B}_{i}\epsilon_{y_{i}, t}^{2} + \hat{C}_{i}\epsilon_{z_{i}, t}^{2} + \hat{D}_{i}\epsilon_{x_{i}, t}\epsilon_{y_{i}, t} + \hat{E}_{i}\left(Z_{i, t}\right)\epsilon_{x_{i}, t} + \hat{F}_{i}\left(X_{i, t}\right)\epsilon_{z_{i}, t}\right) + \hat{G}\left(\Xi_{t}\right)\\
\displaystyle \mathrel{\phantom{L\hat{V}\left(\Xi_{t}\right) = }} + \sum_{\substack{i, j = 1\\ i \neq j}}^{2}\hat{H}_{i}\epsilon_{x_{j}, t}\epsilon_{y_{i}, t},
\end{array}
\end{equation*}
\noindent where each coefficient is the counterpart of the ones of Theorem~\ref{th}, that is, for $i = 1, 2$,
\begin{itemize}
\item $\hat{A}_{i}\left(X_{i, t}\right) = \tau\left(b_{1} - X_{i, t}^{2} - X_{i, t}x_{i}^{\ast} - \left(x_{i}^{\ast}\right)^{2}\right)$.

\item $\hat{B}_{i} = \tau\varepsilon k_{i}a_{1}$.

\item $\hat{C}_{i} = - \frac{\tau\varepsilon}{\tau_{z}}$.

\item $\hat{D}_{i} = \tau\left(b_{0} + \varepsilon k_{i}\left(a_{0} + c\right)\right)$.

\item $\hat{E}_{i}\left(Z_{i, t}\right) = - \tau\left(\phi_{f}\left(Z_{i, t}\right) - \phi_{f}\left(z_{i}^{\ast}\right)\right)$.

\item $\hat{F}_{i}\left(X_{i, t}\right) = \tau\varepsilon\left(\phi_{r}\left(X_{i, t}\right) - \phi_{r}\left(x_{i}^{\ast}\right)\right)$.

\item $\hat{G}\left(\Xi_{t}\right) = \frac{1}{2}G_{1}\left(\Xi_{t}\right)^{2} + \frac{1}{2}\varepsilon G_{2}\left(\Xi_{t}\right)^{2} + \frac{1}{2}G_{3}\left(\Xi_{t}\right)^{2} + \frac{1}{2}\varepsilon G_{4}\left(\Xi_{t}\right)^{2}$.

\item $\hat{H}_{i} = - \tau\varepsilon k_{i}c$.
\end{itemize}

\noindent Using Young's inequality \eqref{ine} exactly as in Theorem~\ref{th}, all mixed terms can be estimated in the same way. The only additional contributions with respect to the three-dimensional case are those arising from the coupling, namely
\begin{equation*}
\hat{D}_{1}\varepsilon_{x_{1}, t}\varepsilon_{y_{1}, t} + \hat{D}_{2}\varepsilon_{x_{2}, t}\varepsilon_{y_{2}, t} + \hat{H}_{1}\varepsilon_{x_{2}, t}\varepsilon_{y_{1}, t} + \hat{H}_{2}\varepsilon_{x_{1}, t}\varepsilon_{y_{2}, t}.
\end{equation*}

\noindent Applying again Young's inequality \eqref{ine} to these terms, and using the dissipativity provided by $a_{1}$ in the equations of $Y_{1}$ and $Y_{2}$, and by $- \frac{1}{\tau_{z}}$ in the equations of $Z_{1}$ and $Z_{2}$, we obtain negative coefficients in the terms $\varepsilon_{y_{i}, t}^{2}$ and $\varepsilon_{z_{i}, t}^{2}$, $i = 1, 2$, provided that $K$ satisfies
\begin{equation*}
K < \frac{\tau\varepsilon}{2}\min\left\{- k_{1}a_{1}, - k_{2}a_{1}, \frac{1}{\tau_{z}}\right\}.
\end{equation*}

\noindent Therefore, proceeding exactly as in the proof of Theorem~\ref{th}, we conclude that the only potentially nonnegative contributions are the parabolic terms in $\varepsilon_{x_{1}, t}$ and $\varepsilon_{x_{2}, t}$. Since these are concave parabolas, each of them can be decomposed into a positive part and a negative part. Repeating the same argument as in Theorem~\ref{th}, this yields an estimate of the form
\begin{equation*}
L\hat{V} \leq C_{2}\left\|\Xi_{t} - \Xi^{\ast}\right\|^{2} + C_{1},
\end{equation*}
\noindent where $C_{1} \geq 0$ and $C_{2} < 0$.

\noindent Substituting the above inequality into \eqref{dhatV}, integrating between $0$ and $t$, taking expectations, and proceeding exactly as in Theorem~\ref{th}, we finally obtain
\begin{equation*}
\limsup_{t \to \infty}\frac{1}{t}\mathbb{E}\left[\int_{0}^{t}\left\|\Xi_{s} - \Xi^{\ast}\right\|^{2}\mathrm{d}s\right] \leq C,
\end{equation*}
\noindent which proves \eqref{condi6d}.
\end{proof}

\begin{remark}\label{remfin}
As in Remark~\ref{remark2}, the ideal situation would be that, whenever the equilibrium point is stable in the deterministic setting, the constant in \eqref{condi6d} were equal to zero. As in the three-dimensional case, the estimate obtained here shows that the deviation with respect to the equilibrium is governed by the corresponding parabolic terms, which now depend on both $x_{1}^{\ast}$ and $x_{2}^{\ast}$.

As will be illustrated in the numerical section, in the regime of total oscillation death, where the deterministic solution converges globally to the equilibrium, the stochastic trajectories may still exhibit occasional jumps. However, consistently with Remark~\ref{remark2}, such events become less likely as $\left|x_{i}^{\ast}\right|$, $i = 1, 2$, increases and as the noise intensity $K$ decreases.
\end{remark}

\subsection{Numerical simulations of the stochastic coupled model}
We proceed as in Section~\ref{sec4}. Here we illustrate numerically the qualitative behaviours of the stochastic coupled ICC system by studying the linear stochastic model \eqref{lineal6d}. To isolate the role of each stochastic perturbation, we activate only one noise intensity at a time; that is, we take $\sigma_{i} \neq 0$, $i = 1, \dots, 5$, for a single coefficient while setting the remaining ones equal to zero.

Unless otherwise stated, the initial condition is chosen as at the beginning of this section. As in the three-dimensional case, stochastic perturbations may induce intermittent large excursions in the attractive regime, destroy the strict periodicity of deterministic oscillations, and generate transitions between different dynamical patterns.

\subsubsection{Total oscillation death}
As discussed in Remark~\ref{remfin}, even in the regime of total oscillation death stochastic perturbations may induce occasional jumps. In particular, a realization of the Wiener process may drive the solution away from the stable equilibrium, producing a large excursion before returning to its neighbourhood.

We first illustrate the smallest values of the noise intensities for which jumps are observed within the time interval $\left[0, T\right]$. Figure~\ref{Aspike} shows the corresponding simulations when a single noise component is activated at a time. As in the three-dimensional case, the trajectories remain close to the equilibrium for long intervals but sporadically perform large excursions.

These values correspond to the smallest noise intensities for which a jump is observed under the fixed realizations of the Wiener processes described above. Depending on the realization and on the observation window, jumps may occur earlier, later, or not at all.

\begin{figure}[ht]
\centering
\begin{subfigure}[ht]{0.45\textwidth}
\centering
\includegraphics[scale=0.28]{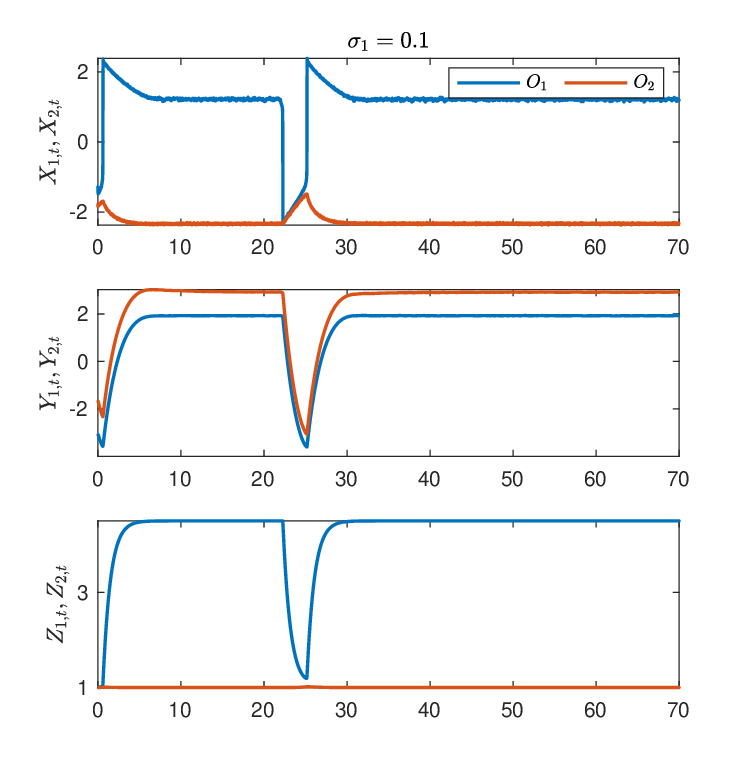}
\caption{$\sigma_{1} = 0.1$, $\sigma_{j} = 0$ for $j \neq 1$.}
\end{subfigure}
\hfill
\begin{subfigure}[ht]{0.45\textwidth}
\centering
\includegraphics[scale=0.28]{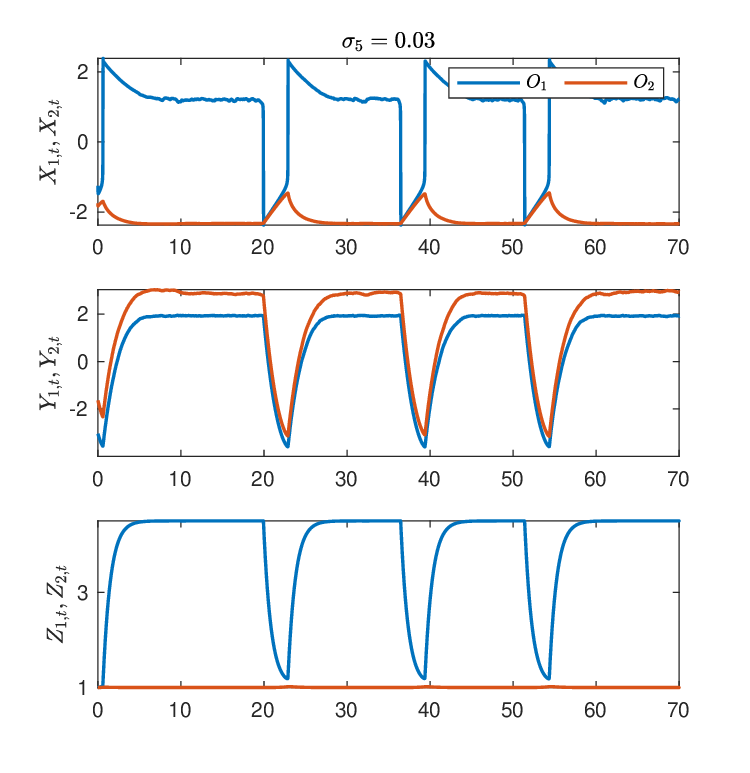}
\caption{$\sigma_{5} = 0.03$, $\sigma_{j} = 0$ for $j \neq 5$.}
\end{subfigure}
\caption{Noise-induced jumps in the total oscillation death regime for model \eqref{lineal6d}.}\label{Aspike}
\end{figure}

To illustrate the effect of the noise level, we double these values in Figure~\ref{Aspikes}. In this case, jumps occur more frequently and the intermittent character of the dynamics becomes more pronounced.

\begin{figure}[ht]
\centering
\begin{subfigure}[ht]{0.45\textwidth}
\centering
\includegraphics[scale=0.28]{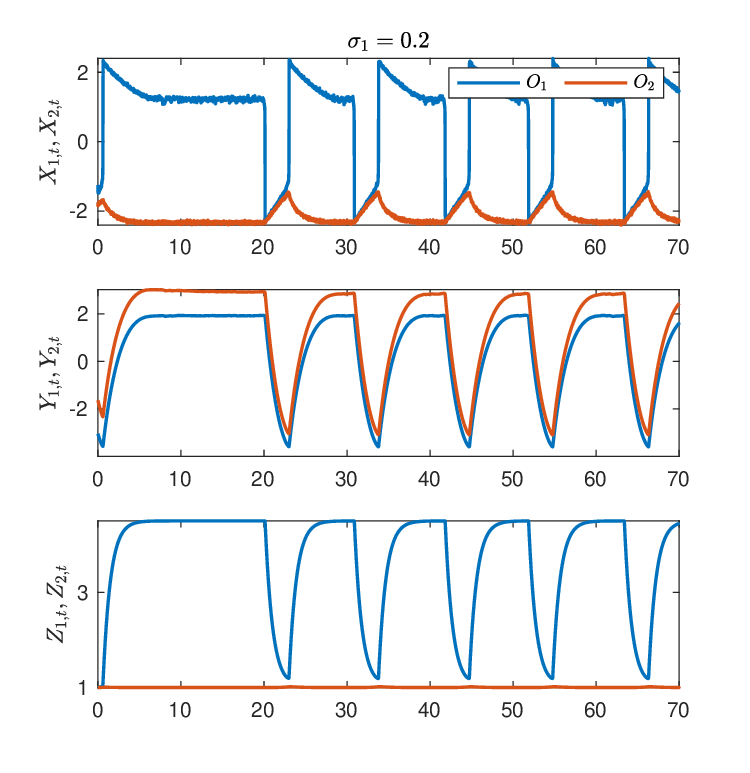}
\caption{$\sigma_{1} = 0.2$, $\sigma_{j} = 0$ for $j \neq 1$.}
\end{subfigure}
\hfill
\begin{subfigure}[ht]{0.45\textwidth}
\centering
\includegraphics[scale=0.28]{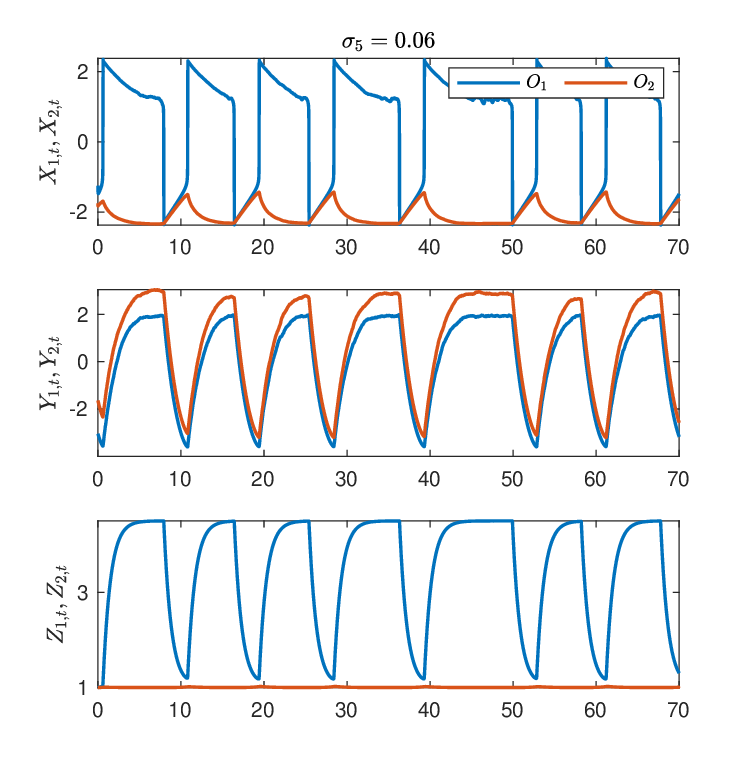}
\caption{$\sigma_{5} = 0.06$, $\sigma_{j} = 0$ for $j \neq 5$.}
\end{subfigure}
\caption{Noise-induced jumps with doubled noise intensities.}\label{Aspikes}
\end{figure}

\subsubsection{Relaxation loss}
In the deterministic relaxation loss regime, cell $O_{1}$ always produces spikes, whereas cell $O_{2}$ exhibits only small-amplitude oscillations. In the stochastic framework, however, different behaviours may occur.

First, it may happen that the spikes are not always produced by the same cell; in particular, the cell that does not spike in the deterministic case may also jump. Second, both cells may spike simultaneously, as illustrated in Figure~\ref{B}.

This situation is related to the antiphase synchronization regime, in which the cells spike alternately. However, in the present case the dynamics is less structured, and the asymmetry observed in the deterministic framework is generally lost, so that either cell may produce the jump.

\begin{figure}[ht]
\centering
\begin{subfigure}[ht]{0.45\textwidth}
\centering
\includegraphics[scale=0.28]{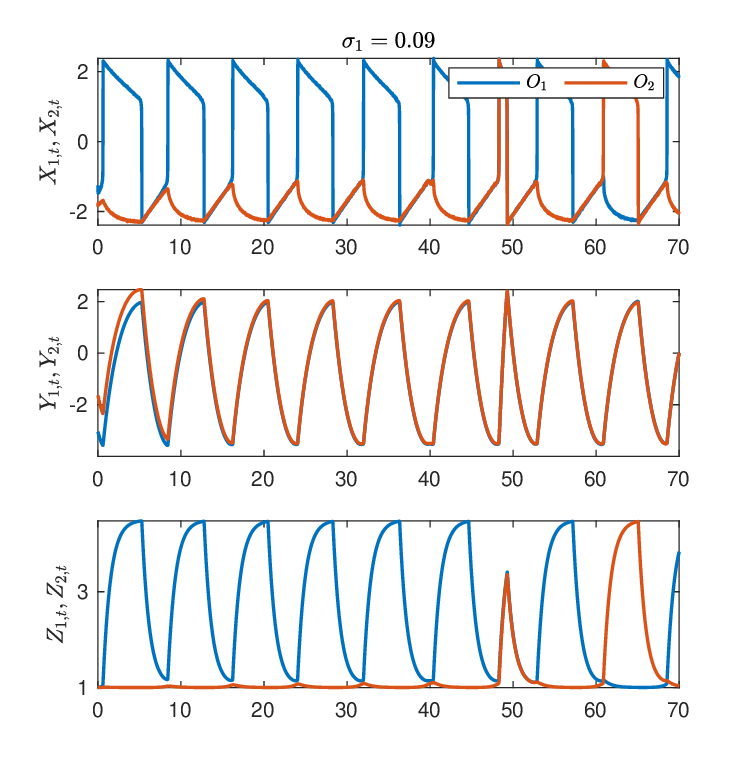}
\caption{$\sigma_{1} = 0.09$, $\sigma_{j} = 0$ for $j \neq 1$.}
\end{subfigure}
\hfill
\begin{subfigure}[ht]{0.45\textwidth}
\centering
\includegraphics[scale=0.28]{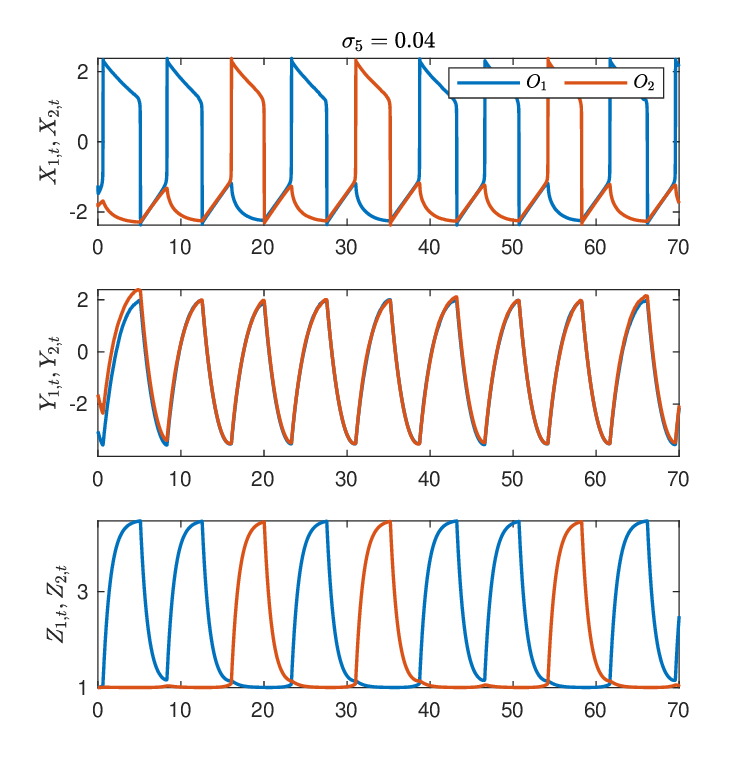}
\caption{$\sigma_{5} = 0.04$, $\sigma_{j} = 0$ for $j \neq 5$.}
\end{subfigure}
\caption{Relaxation loss regime for model \eqref{lineal6d}. In the stochastic framework, either cell may spike and both cells may spike simultaneously.}\label{B}
\end{figure}

\subsubsection{Antiphase synchronization}
To conclude with $c < 0$, we next consider the antiphase synchronization regime. In the deterministic framework, the cells spike alternately periodically and no MMOs are observed.

In the stochastic setting, this alternation is generally lost: the jumps no longer occur periodically and MMOs may appear, as in the three-dimensional case discussed in Subsection~\ref{nume}. Moreover, it may even happen that both cells spike almost simultaneously, leading to episodes that resemble in-phase behaviour. These effects are illustrated in Figure~\ref{C}.

\begin{figure}[ht]
\centering
\begin{subfigure}[ht]{0.45\textwidth}
\centering
\includegraphics[scale=0.28]{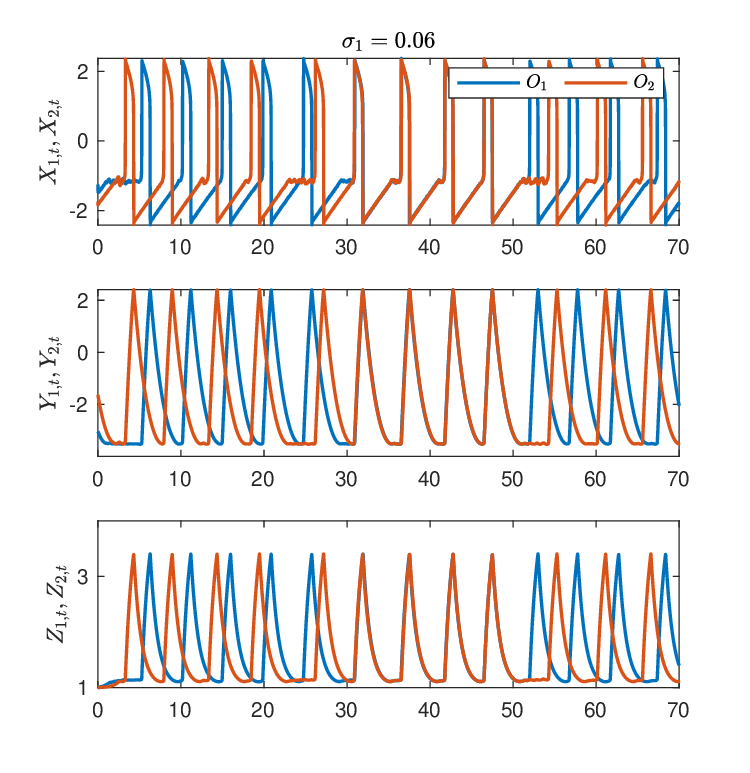}
\caption{$\sigma_{1} = 0.06$, $\sigma_{j} = 0$ for $j \neq 1$.}
\end{subfigure}
\hfill
\begin{subfigure}[ht]{0.45\textwidth}
\centering
\includegraphics[scale=0.28]{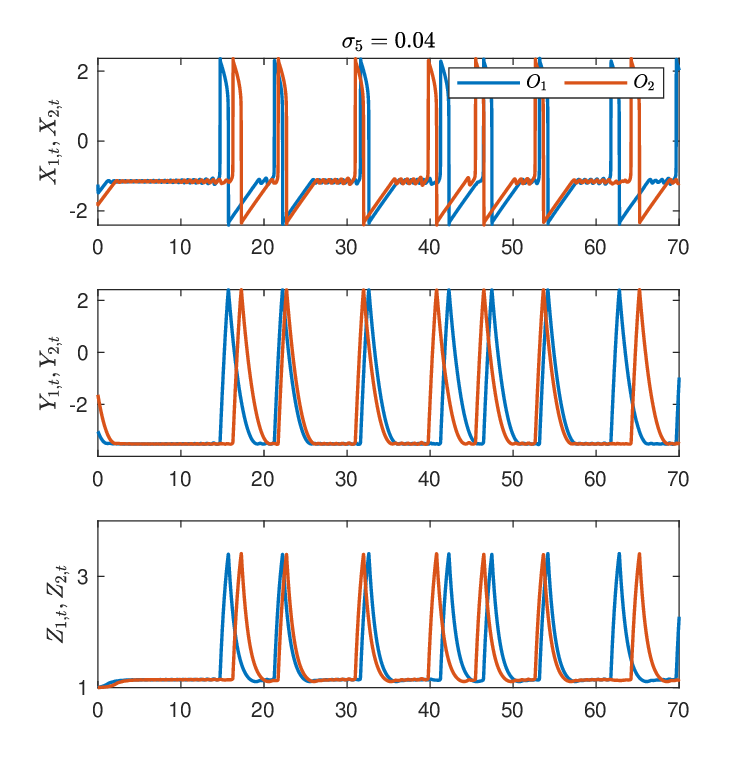}
\caption{$\sigma_{5} = 0.04$, $\sigma_{j} = 0$ for $j \neq 5$.}
\end{subfigure}
\caption{Antiphase synchronization regime for model \eqref{lineal6d}. In the stochastic framework, alternation and periodicity are lost, MMOs may appear, and in-phase episodes may occur.}\label{C}
\end{figure}

\subsubsection{Case $c \geq 0$}
We finally consider together the uncoupled case, the almost-in-phase synchronization regime and the in-phase locking synchronization regime. In the stochastic framework, the three cases tend to evolve towards in-phase behaviour. Moreover, the larger the value of $c$, the faster this synchronization is observed.

As in the three-dimensional case, when MMOs appear they are no longer periodic. This behaviour is consistent with the fact that both neurons have identical parameters, which are affected by the stochastic perturbations, as illustrated in Figure~\ref{F}.

We recall that, for different values of $k_{i}$, $i = 1, 2$, and for different values of the coupling parameter $c$, studied in \cite{Bandera2022, Bandera2026}, respectively, further dynamical behaviours may arise, some of which are already suggested by the numerical simulations presented in this section. The stochastic analysis of these cases is left for future work.

\begin{figure}[ht]
\centering
\begin{subfigure}[ht]{0.95\textwidth}
\centering
\begin{subfigure}[ht]{0.45\textwidth}
\centering
\includegraphics[scale=0.28]{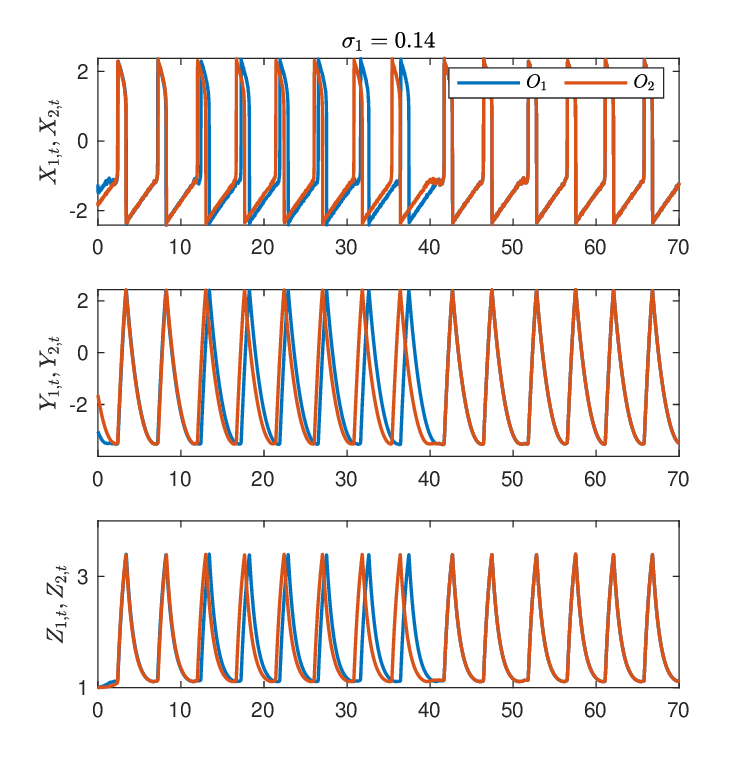}
\end{subfigure}
\hfill
\begin{subfigure}[ht]{0.45\textwidth}
\centering
\includegraphics[scale=0.28]{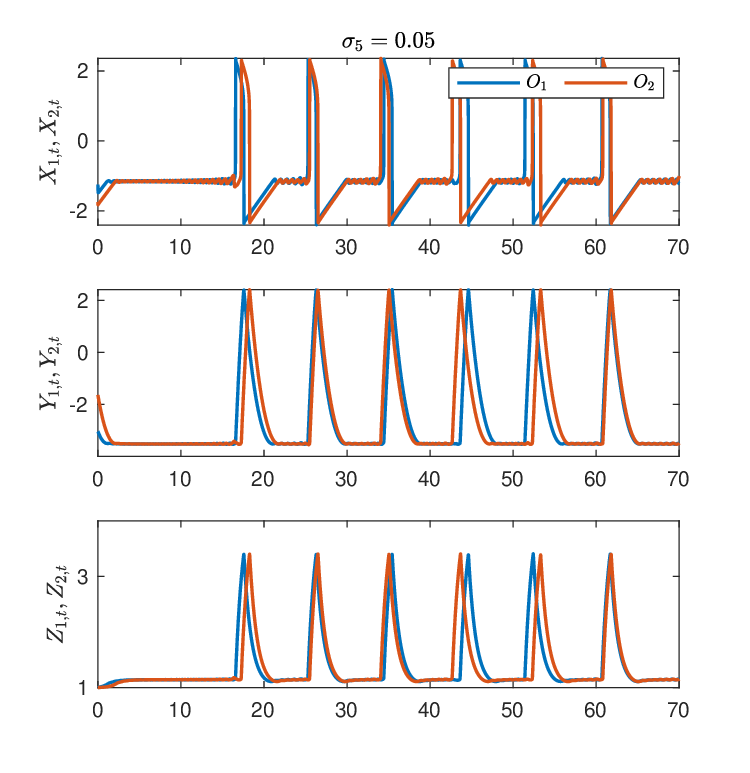}
\end{subfigure}
\caption{Uncoupled regime for model \eqref{lineal6d}. In the stochastic framework, the trajectories tend towards an in-phase behaviour. Left: $\sigma_{1} = 0.1$. Right: $\sigma_{5} = 0.06$.}
\end{subfigure}
\begin{subfigure}[ht]{0.95\textwidth}
\centering
\begin{subfigure}[ht]{0.45\textwidth}
\centering
\includegraphics[scale=0.28]{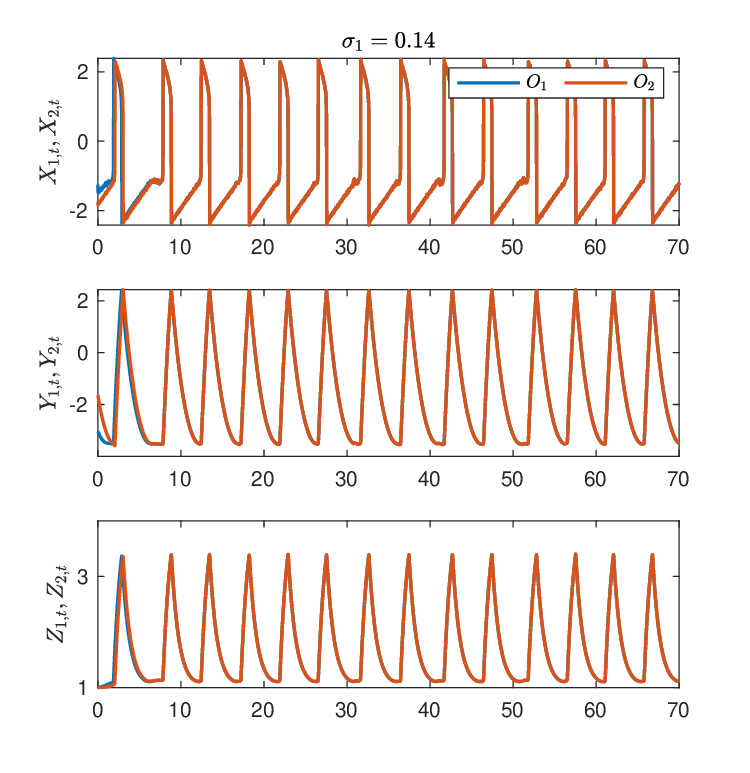}
\end{subfigure}
\hfill
\begin{subfigure}[ht]{0.45\textwidth}
\centering
\includegraphics[scale=0.28]{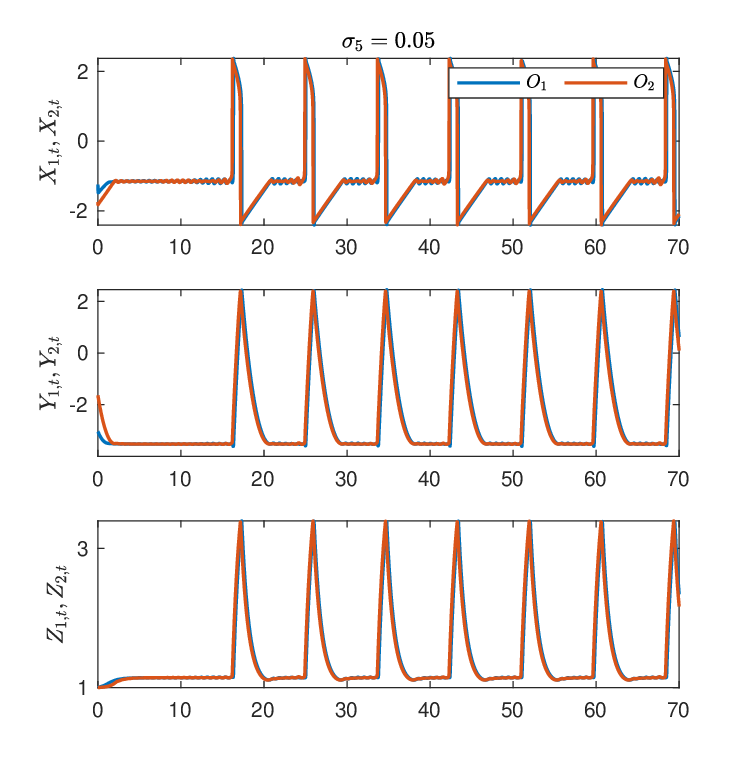}
\end{subfigure}
\caption{Almost-in-phase synchronization regime for model \eqref{lineal6d}. In the stochastic framework, the trajectories tend towards an in-phase behaviour. Left: $\sigma_{1} = 0.1$. Right: $\sigma_{5} = 0.06$.}
\end{subfigure}
\begin{subfigure}[ht]{0.95\textwidth}
\centering
\begin{subfigure}[ht]{0.45\textwidth}
\centering
\includegraphics[scale=0.28]{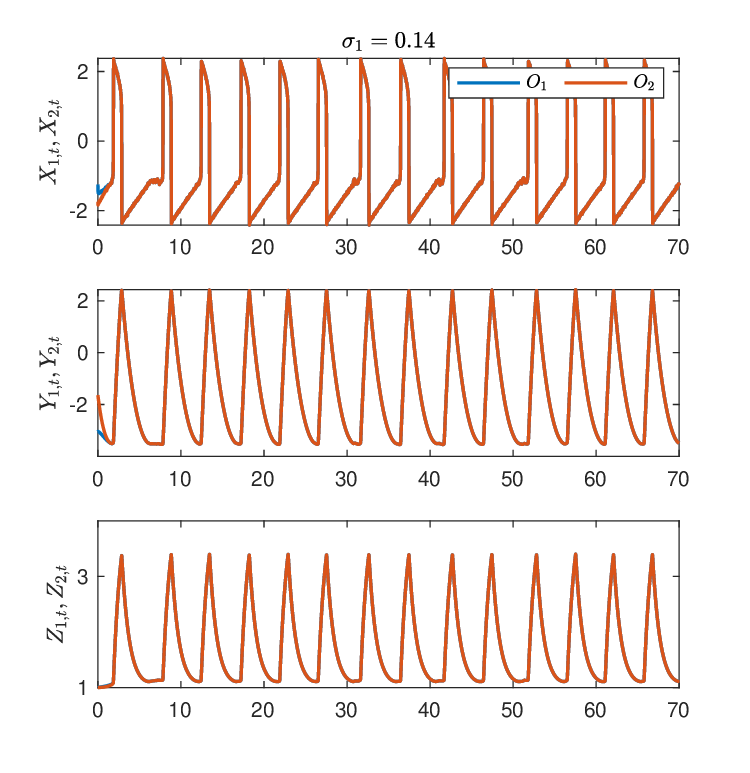}
\end{subfigure}
\hfill
\begin{subfigure}[ht]{0.45\textwidth}
\centering
\includegraphics[scale=0.28]{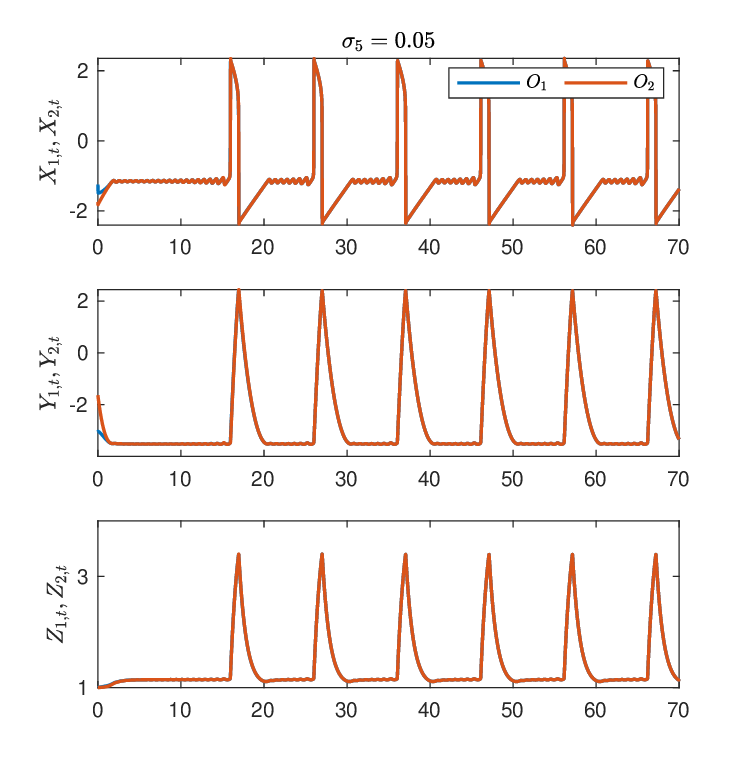}
\end{subfigure}
\caption{In-phase locking synchronization regime for model \eqref{lineal6d}. In the stochastic framework, the trajectories remain in phase. Left: $\sigma_{1} = 0.1$. Right: $\sigma_{5} = 0.06$.}
\end{subfigure}
\caption{Qualitative behaviour of model \eqref{lineal6d} for $c \geq 0$. In the stochastic framework, the dynamics consistently tends towards in-phase synchronization in all three regimes (uncoupled, almost-in-phase, and in-phase locking).}\label{F}
\end{figure}

\subsection{Numerical validation of the theoretical estimates}
To conclude this section, we proceed as in the three-dimensional case. In particular, we approximate numerically the quantity in \eqref{condi6d} using the empirical estimator \eqref{Ctn}, now considering the six-dimensional variable $\Xi_{t}$.

We focus on the total oscillation death regime, taking $c = - 0.514$ and $\mu = 2.4$, as in the previous simulations. In the deterministic framework, we obtain $\hat{C}_{T} \approx 0.095946$, which is small and consistent with the theoretical prediction that the limit in \eqref{condi6d} vanishes when the equilibrium is globally asymptotically stable.

We next perform the same computation in the stochastic framework, considering $M = 50$ independent realizations and increasing values of the noise intensities $\sigma_{1}$ and $\sigma_{5}$, with a single active noise source in each simulation. The results are shown in Figure~\ref{fig:stochC6d}.

Overall, the results are consistent with the theoretical estimate in \eqref{condi6d}, where the bound depends explicitly on the noise intensity.

\begin{figure}[ht]
\centering
\begin{subfigure}[ht]{0.3\textwidth}
\centering
\includegraphics[scale=0.32]{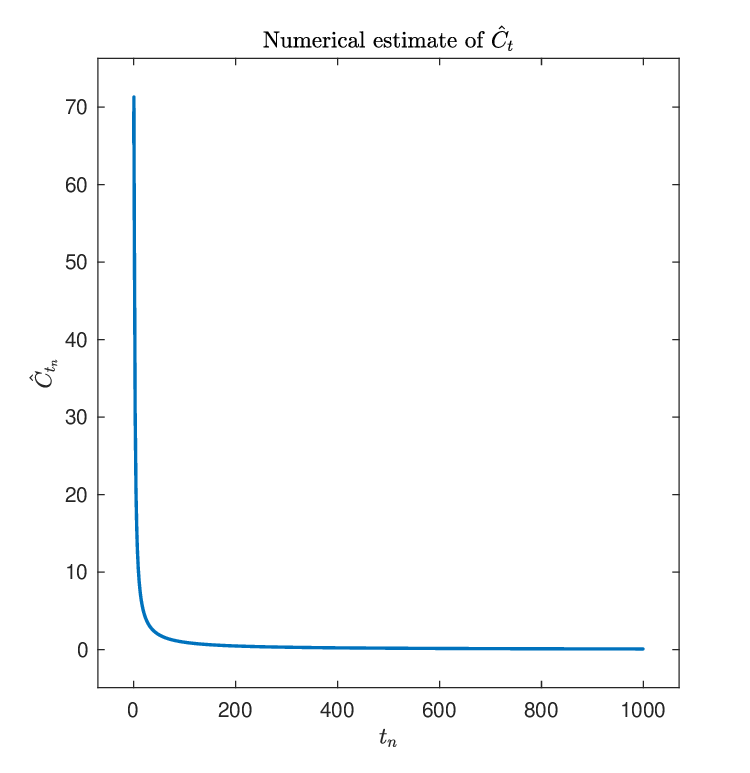}
\caption{Deterministic case.}
\end{subfigure}
\hfill
\begin{subfigure}[ht]{0.3\textwidth}
\centering
\includegraphics[scale=0.32]{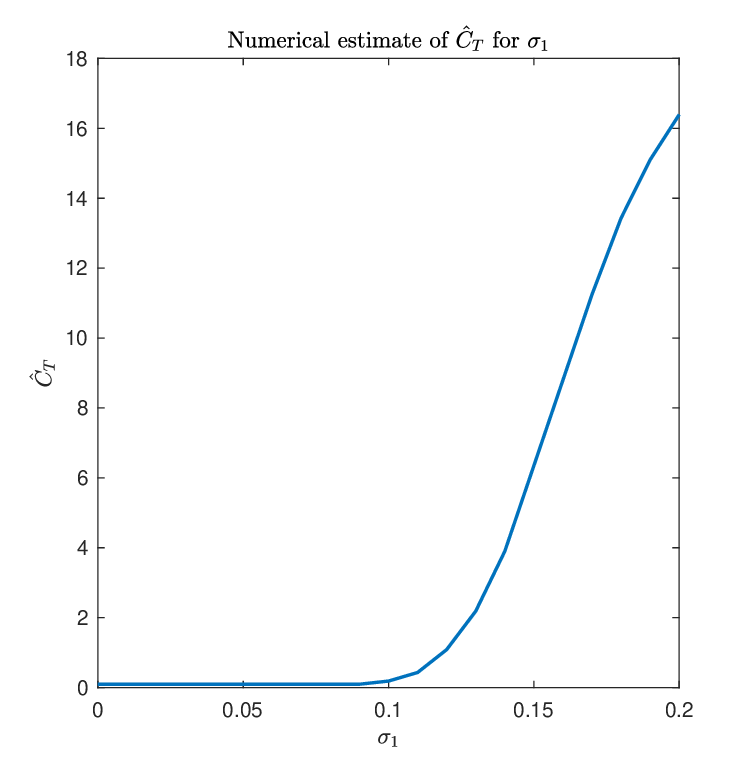}
\caption{$\sigma_{1} \in \left[0, 0.2\right]$.}
\end{subfigure}
\hfill
\begin{subfigure}[ht]{0.3\textwidth}
\centering
\includegraphics[scale=0.32]{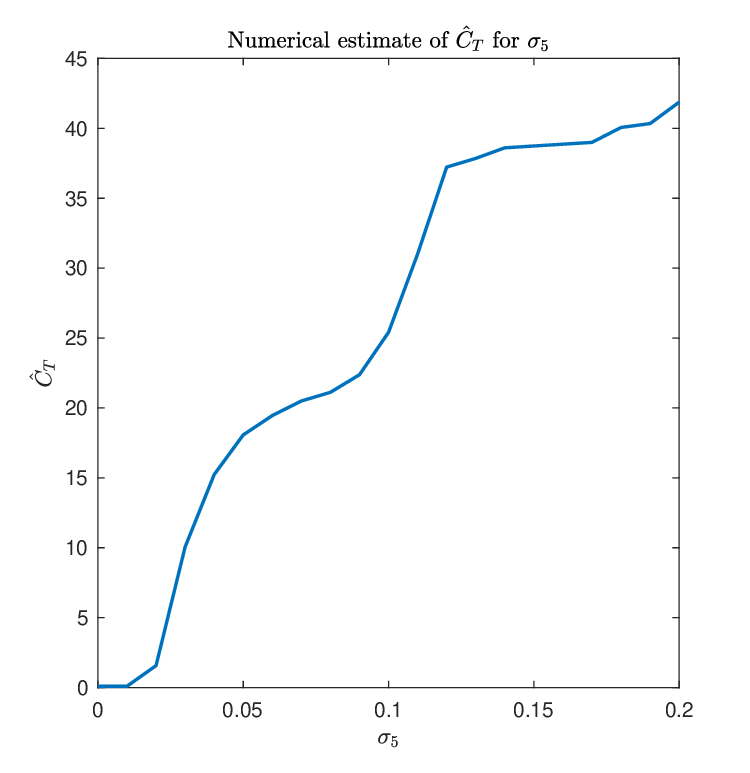}
\caption{$\sigma_{5} \in \left[0, 0.2\right]$.}
\end{subfigure}
\caption{Numerical approximation of \eqref{condi6d} for the stochastic model \eqref{lineal6d} in the total oscillation death regime, computed from $M = 50$ realizations and increasing noise intensities.}\label{fig:stochC6d}
\end{figure}

\section*{Conclusions}
In this paper, we have investigated stochastic versions of a slow-fast model for Intracellular Calcium Concentrations (ICC) based on the FitzHugh-Nagumo structure. Starting from the three-dimensional single-cell formulation, we introduced stochastic perturbations preserving the slow-fast dynamics and analysed the resulting stochastic differential equations from both theoretical and numerical points of view.

For the stochastic single-cell model, we established existence and uniqueness of solutions and proved positivity of the calcium variable. We also derived mean-square asymptotic estimates describing the behaviour of stochastic trajectories around deterministic equilibria. These estimates quantify explicitly the influence of the noise intensity and show that stochastic solutions remain close, on average, to the deterministic dynamics whenever the perturbations are sufficiently small.

The numerical simulations revealed several qualitative effects induced by noise. In the attractive regime, stochastic perturbations may generate intermittent large excursions that are absent in the deterministic framework. In the oscillatory regimes, noise destroys the strict periodicity of deterministic mixed-mode oscillations and may induce mixed-mode patterns in parameter regions where the deterministic system exhibits only relaxation oscillations.

We then extended the analysis to a coupled six-dimensional two-cell ICC model. For this stochastic coupled system, we again established positivity, existence and uniqueness of solutions, together with asymptotic estimates around deterministic equilibria. Numerical simulations showed that stochastic perturbations may alter synchronization patterns, generate transitions between dynamical regimes, and induce qualitative behaviours not present in the deterministic setting, including irregular synchronization episodes and noise-induced jumps in the total oscillation death regime.

Overall, the results obtained in this work show that stochastic perturbations play a significant role in the dynamics of ICC models, enriching the behaviour of the deterministic system while preserving its underlying slow-fast structure. Several problems remain open for future research, including the analysis of heterogeneous networks with different coupling structures, the study of larger collections of interacting cells, and the investigation of stochastic bifurcation phenomena in coupled ICC systems.

\paragraph*{Acknowledgement}
This document is the result of the research project funded by The Andalusian Government under grant PREDOC-PAID2020, The Spanish Ministerio de Ciencia e Innovaci{\'o}n, Agencia Estatal de Investigaci{\'o}n (AEI) and FEDER under projects PID2024-156228NB-I00 and PID2021-123153OB-C21, The National Natural Science Foundation of China (Grant No.~12471165), and The Generalitat Valenciana, Conselleria de Educaci{\'o}n, Cultura, Universidades y Empleo, Grant CIAICO/2024/251.

\bibliographystyle{plain}
\bibliography{references}
\end{document}